\documentclass[12pt,a4paper]{amsart}
\usepackage{amssymb}
\usepackage{amsfonts}
\usepackage{amsmath}
\usepackage[mathscr]{eucal}

\newtheorem{thm}{Theorem}[section]
\newtheorem{prop}[thm]{Proposition}
\newtheorem{lem}[thm]{Lemma}

\numberwithin{equation}{section}

\def\N{{\Bbb N}}
\def\Z{{\Bbb Z}}
\def\Q{{\Bbb Q}}
\def\R{{\Bbb R}}
\def\C{{\Bbb C}}
\def\A{{\Bbb A}}

\def\emp{\varnothing}

\def\fa{{\frak a}}

\def\fn{{\frak n}}

\def\fF{{\frak F}}
\def\fH{{\frak H}}
\def\fE{{\frak E}}

\def\fX{{\frak X}}

\def\fP{{\frak P}}

\def\cO{\frak o}

\def\cT{{\mathcal T}}

\def\cW{{\mathcal W}}

\def\GL{{\operatorname {GL}}}

\def\SL{{\operatorname{SL}}}
\def\SO{{\operatorname{SO}}}

\def\PSL{{\operatorname {PSL}}}

\def\Re{{\operatorname {Re}}}
\def\Im{{\operatorname {Im}}}

\def\nr{{\operatorname{N}}}
\def\Mat{{\operatorname{M}}}

\def\diag{{\operatorname {diag}}}
\def\sgn{{\operatorname {sgn}}}
\def\Ad{{\operatorname{Ad}}}
\def\vol{{\operatorname{vol}}}

\def\leq{\leqslant}
\def\geq{\geqslant}
\def\bsl{\backslash}
\def\le{\leq}
\def\ge{\geq}

\def\d {{{d}}}

\def\JJ{{\Bbb J}}

\def\bK{{\bold K}}

\def\1{{\bold 1}}

\renewcommand{\a}{\alpha}
\renewcommand{\b}{\beta}

\newcommand{\e}{\epsilon}

\renewcommand{\l}{\lambda}

\newcommand{\go}{{\mathfrak{o}}}

\newcommand{\Dcal}{{\mathcal D}}

\newcommand{\Fcal}{{\mathcal F}}

\newcommand{\Ocal}{{\mathcal O}}
\newcommand{\Pcal}{{\mathcal P}}
\newcommand{\Qcal}{{\mathcal Q}}

\newcommand{\Scal}{{\mathcal S}}
\newcommand{\Tcal}{{\mathcal T}}

\renewcommand{\AA}{\mathbb{A}}

\newcommand{\CC}{\mathbb{C}}

\newcommand{\II}{\mathbb{I}}

\newcommand{\NN}{\mathbb{N}}
\newcommand{\PP}{\mathbb{P}}
\newcommand{\QQ}{\mathbb{Q}}
\newcommand{\RR}{\mathbb{R}}

\newcommand{\ZZ}{\mathbb{Z}}

\newcommand{\bfs}{{\mathbf s}}

\newcommand{\bfA}{{\mathbf A}}

\newcommand{\bfB}{{\mathbf B}}

\newcommand{\ord}{\operatorname{ord}}

\newcommand{\fin}{{\rm fin}}
\renewcommand{\Re}{\operatorname{Re}}

\newcommand{\Res}{\operatorname{Res}}

\def\cT{{\mathcal T}}

\title{Quantitative non-vanishing of central values of certain $L$-functions on $\GL(2)\times\GL(3)$}

\author{Shingo Sugiyama}
\address{
		Department of Mathematics, College of Science and Technology, Nihon University,
		1-8-14, Suruga-Dai, Kanda, Chiyoda, Tokyo 101-8308, Japan} 
\email{sugiyama.shingo@nihon-u.ac.jp}

\author{Masao Tsuzuki}
\address{
	Department of Science and Technology, Sophia University, Kioi-cho 7-1 Chiyoda-ku Tokyo, 102-8554, Japan} 
\email{m-tsuduk@sophia.ac.jp}


\begin{document}
	\maketitle
\begin{abstract}
Let $\phi$ be an even Hecke-Maass cusp form on $\SL_2(\ZZ)$ whose $L$-function does not vanish at the center of the functional equation. In this article, we 
obtain an exact formula of the average of triple products of $\phi$, $f$ and $\bar f$, where $f$ runs over an orthonormal basis $H_k$ of Hecke eigen elliptic cusp forms on $\SL_2(\Z)$ of a fixed weight $k\geq 4$. As an application, we prove a quantitative non-vanishing results on the central values for the family of degree $6$ $L$-functions $L(s,\phi \times {\rm Ad}\,f)$ with $f$ in the union of $H_k$ $({\rm K}\leq k<2{\rm K})$ as ${\rm K}\rightarrow \infty$. 
\end{abstract}
\setcounter{tocdepth}{1}

\section{Introduction} \label{Sec1}

\subsection{Background and motivation}
Let $\phi(\tau)$ be a Hecke-Maass cusp form on $\SL_2(\Z)$, i.e., an $\SL_2(\Z)$-invariant real analytic cusp form on the Poincar\'{e} upper-half plane $\fH=\{\tau=x+iy|x\in \R,\,y>0\}$
which is a joint eigenfunction of all the Hecke operators and the hyperbolic Laplacian $\Delta=-y^2(\tfrac{\partial^2}{\partial x^2}+\tfrac{\partial^2}{\partial y^2})$. By a well-known lower estimate of the first eigenvalue of $\Delta$ on $\SL_2(\Z)\bsl \fH$, we may write the Laplace eigenvalue of $\phi$ as $(1-\nu_\infty^2)/4$ with $\nu_\infty\in i\R$ (cf.\ \cite{Roelcke}, \cite{Hejhal}).
In what follows, we suppose that $\phi(\tau)$ is even, i.e., $\phi(-\bar\tau)=\phi(\tau)$, and that $\phi$ is normalized so that its first Fourier coefficient is $1$, i.e., 
\begin{align}\phi(\tau)=\sum_{n\in \Z-\{0\}}\lambda_{\phi}(n)\,y^{1/2}K_{\nu_\infty/2}(2\pi|n|y)\,e^{2\pi i nx}, \quad \lambda_{\phi}(1)=1.
 \label{phiNormal}
\end{align}
 Let $L(s,\phi)$ be the Hecke $L$-function of $\phi$ defined as the holomorphic continuation of the Euler product of degree $2$
 $$ L(s,\phi)=\prod_{p}\det(1_2-A_p(\phi)p^{-s})^{-1}, \quad \Re s>1,
$$
where $A_p(\phi) \in \GL_2(\C)$ is the Satake parameter of $\phi$ at a prime number $p$. Note that ${\rm tr}\,A_{p}(\phi)=\lambda_{\phi}(p)$. The completed $L$-function $\hat L(s,\phi)=\Gamma_{\R}(s+\nu_\infty/2)\Gamma_{\R}(s-\nu_\infty/2)L(s,\phi)$ satisfies the self-dual functional equation $\hat L(1-s,\phi)=\hat L(s,\phi)$. Although it seems to be a common belief that the central value $L\left(1/2,\phi\right)$ is always non-zero, no proof is given so far. To our best knowledge, the infinitum of $\phi$ with $L(1/2,\phi)\not=0$ is shown by an asymptotic formula \cite[Theorem 2]{Motohashi}. Contrastingly, no single $\phi$ with $L(1/2,\phi)=0$ is known to exist, whereas it is shown that such $\phi$, if any, has to satisfy a very strong vanishing property of the toric period integrals (\cite{Zagier81}, \cite{KatokSarnak}). For an even positive integer $k$, let $S_k(\SL_2(\Z))$ be the space of holomorphic cusp forms on $\SL_2(\Z)$ of weight $k$ endowed with the Petersson inner product
$$
\langle f, f_1 \rangle=\int_{\SL_2(\ZZ)\bsl\mathfrak{H}}f(\tau)\bar f_1(\tau)\,(\Im \tau)^{k}d\mu(\tau), \qquad f,\,f_1 \in S_k(\SL_2(\Z)). 
$$ 
We fix an orthonormal basis $H_k$ of $S_k(\SL_2(\Z))$ consisting of joint eigenfunctions of all the Hecke operators $\{T(n)\}_n$, i.e., 
$$ 
 T(n)f=n^{(k-1)/2}\lambda_f(n)f, \quad (n\in \NN, f\in H_k),
$$
where $\d\mu(\tau)=y^{-2}\d x\d y$ is the hyperbolic volume element of $\mathfrak{H}$. The factor $n^{(k-1)/2}$ is separated so that Deligne's bound takes the form $|\l_f(n)|\le d(n)$ for all $n\in \N$, where $d(n)=\sum_{0<d|n}1$ is the divisor function.

Depending on $f\in H_k$, we consider a probability measure on the modular surface $Y=\SL_2(\Z)\bsl \fH$ defined as  
$$
\mu_{f}(\psi)=\int_{\SL_2(\Z) \bsl \fH} {\psi(\tau)}\, |f(\tau)|^2 \,(\Im \tau)^{k}\d\mu(\tau)
$$
for any bounded measurable function $\psi$ on the surface $Y$. Note that $\vol(Y)=\pi/3$. The holomorphic analogue of Quantum Unique Ergodicity conjecture originally proposed by Rudnick and Sarnak (\cite{RudSar}) asserts that the measure $\mu_{f}$ for every $f\in H_k$ should converge to the probability measure $\vol(Y)^{-1}\mu$ on $Y$ as the weight $k$ grows to infinity. Before the holomorphic QUE conjecture, as well as the original QUE conjecture itself, has been proved (\cite{Sound}, \cite{Linden}, \cite{HolowinskiSoundararajan}), several research concerning the average of the family $\mu_{f}\,(f\in H_k)$ are conducted. For example, Luo \cite{Luo} showed the holomorphic QUE conjecture on average proving the limit formula  
\begin{align}
\frac{1}{\# H_k}\,\sum_{f\in H_k} \mu_{f}(\psi) \quad \rightarrow \quad \vol(Y)^{-1}\mu(\psi), \quad \psi \in C_{0}^{\infty}(Y) \label{Luo}
\end{align}
for the full level case\footnote{As a matter of fact, in \cite{Luo} this formula is shown only when $\psi$ is the characteristic function of a bounded measurable set of $Y$. The proof works for general $\psi$.}. The limiting behavior as ${\rm K}\rightarrow +\infty$ of the ``quantum variance'' 
\begin{align}
\sum_{k\in 2\N} u\left(\frac{k-1}{{\rm K}}\right)\,\sum_{f\in H_k} |\mu_{f}(\psi)|^{2}
 \label{LuoSar1}
\end{align}
and the modified version  
\begin{align}
\sum_{k\in 2\N} u\left(\frac{k-1}{{\rm K}}\right)\,\sum_{f\in H_k}L(1,f,{\rm sym}^2)\, |\mu_{f}(\psi)|^{2}, 
 \label{LuoSar2}
\end{align}
are considered by Luo-Sarnak (\cite{LuoSarnak}), where $u$ is an arbitrary function from $C_{{\rm c}}^{\infty}(0,+\infty)$ and $L(s,f,{\rm sym}^2)$ is the symmetric square $L$-function of $f$. In the above formulas, the ``test function'' $\psi$ is supposed to be taken from the space $C_{0}^{\infty}(Y)$ (resp. $C_{0,0}^{\infty}(Y)$) of all the smooth functions (resp. those with zero means $\int_{\SL_2(\Z)\bsl \fH}\psi(\tau)\, \d \mu(\tau)=0$) which as well as their derivatives satisfy the bound $|\psi(\tau)|\ll_{A} y^{-A}$ on $y>\sqrt{3}/{2}$ for any $A>0$. They showed that (a) there exists a certain Hermitian form $B_{\omega}(\psi)$ on the space $C_{0,0}^\infty(Y)$ such that \eqref{LuoSar2} is asymptotically equal to 
$$
B_\omega(\psi)\, \biggl(\int_{0}^{\infty}u(t)\d t\biggr)\,{\rm K}+
O_{\e,\psi}({\rm K}^{1/2+\e})
$$
and (b) the quantity $B_{\omega}(\phi)$ for an even Hecke-Maass cusp form $\phi$ as above is given as $\frac{\pi}{2}\,L(\frac{1}{2},\phi)\,\|\phi\|^2$. They gave a brief remark indicating that a similar asymptotic formula for the average \eqref{LuoSar1} should hold true. However, the necessary argument is not that direct. After a while, some detail concerning to the proof appeared in \cite{SarnakZhaoWoo}. Strictly speaking \cite{SarnakZhaoWoo} considers the variance for the family $\mu_{\phi_j}$ over an orthonormal system of Hecke-Maass forms $\{\phi_j\}$ in the realm of the original QUE conjecture; we can check that the argument is easily carried over to the holomorphic case. If we choose $u$ to be an appropriate approximation of the characteristic function of the interval $[1,2]$, then the arguments end up with the following asymptotic formula (cf.\ \cite[Corollary 1]{SarnakZhaoWoo}). 

\begin{thm}(Luo-Sarnak) \label{thmLuoSar}
\begin{align}
\sum_{k\in 2\N \cap[{\rm K},2{\rm K})}
\, \sum_{f\in H_k} |\mu_{f}(\phi)|^{2} \sim \frac{\pi\,\|\phi\|^2}{2}L\left(\tfrac{1}{2}, \phi\right)\,C(\phi)\,{\rm K} \quad ({\rm K} \rightarrow +\infty).
 \label{LuoSar3}
\end{align}
The quantity $C(\phi)$ is given by the following convergent Euler product over prime numbers:
$$
C(\phi):=\frac{1}{\zeta(2)}\,\prod_{p}\left(1-\frac{p^{-1}{\rm tr}\,A_{p}(\phi)
}{p^{1/2}+p^{-1/2}} \right).
$$
\end{thm}
In this article, we consider the weighted first moment
\begin{align}
{\bf A}_{k,n}(\phi)=\sum_{f\in H_k}\mu_{f}(\phi)\,\lambda_{f}(n), \quad n \in \N \label{Av}
\end{align}
and its signed average
\begin{align}
\frac{1}{2^{-1}{\rm K}}\sum_{k \in 
2\N \cap [{\rm K}, 2{\rm K})}
	(-1)^{k/2}\bfA_{k,n}(\phi)
\label{SAv}
\end{align}
for an even Hecke-Maass cusp form $\phi$ on $\SL_2(\Z)$ and describe an exact formula of \eqref{Av} when $\phi$ is the normalzied Eisenstein series $E^*(z)$ or a cusp form in a uniform way (Theorem~\ref{MVT-L3-1} and \eqref{Zagierformula}). When $\phi$ is cuspidal, we deduce asymptotic formulas of \eqref{Av} and of \eqref{SAv} for any fixed $n\in \N$ with the weight $k$ growing to infinity (Theorems~\ref{ASYMPT} and \ref{ASYMPT_average}). Luo's result \eqref{Luo} applied to an even Hecke-Maass form $\phi$ yields $(\#H_k)^{-1}\sum_{f\in H_k}\mu_{f}(\phi)\rightarrow 0$ as $k\rightarrow \infty$ since $\mu(\phi)=0$ holds by the cuspidality. Our asymptotic formula improves this result in that it tells us a rate of convergence to $0$.   

Luo-Sarnak's result recalled above has an implication to non-vanishing of central values of $L$-functions as follows. For $f \in H_k$, let $\pi_f$ be the cuspidal automorphic representation of $\GL_2(\A_\Q)$ corresponding to $f$ and ${\rm Ad}(\pi_f)$ the Gelbart-Jacquet lift of $\pi_f$ to $\GL_3(\A_\Q)$ (\cite{GelbartJacquet2}); the latter is an irreducible cuspidal representation of $\GL_3(\A_\Q)$ whose Satake parameter at any prime number $p$ is $r_2(A_p(f))$, where $r_2:\GL_2(\C)\rightarrow \GL_3(\C)$ denotes the adjoint representation and $A_p(f) \in \GL_2(\C)$ is the Satake parameter of $f$ at $p$.
 The archimedean component of $\pi_{f}$ and of ${\rm Ad}(\pi_{f})$ are isomorphic to $D_{k}$ and ${\rm Ind}_{P}^{\GL_3(\R)}(D_{2k-1}\boxtimes {\rm sgn})$, respectively, where $P \subset \GL_3(\R)$ is the standard parabolic subgroup with Levi subgroup $\GL_2(\R)\times \R^\times$ and $D_{k}$ is the discrete series of $\GL_2(\R)$ of weight $k$ with trivial central character; we remark that both representations have non-vanishing relative Lie algebra cohomologies (\cite{BW}). We choose a vector-valued Hecke eigen cusp form ${\rm Ad} f$ on $\SL_3(\Z)$ belonging to the minimal ${\rm O}(3)$-type of ${\rm Ad}(\pi_{f})$. 
Then the convolution $L$-function for the pair $(\phi,{\rm Ad} f)$ is defined as the analytic continuation of the degree $6$ Euler product 
$$
L(s,\phi \times {\rm Ad} f)=\prod_{p}\det(1_6-A_p(\phi)\otimes r_2(A_p(f))p^{-s})^{-1}, \quad \Re s>1. 
$$ 
Then \cite[Theorem 1]{LuoSarnak} immediately shows that for a given $\phi$ with $L(1/2,\phi)\not=0$ there exists an inifinite family $f_{j}\in H_{k_j}$ $(j\geq 1)$ with growing weights $k_j$ such that $$
L\left(1/2,\phi \times {\rm Ad} f_j\right)\not=0.
$$
Combined with the infinitum of $\phi$ recalled above, this implies the existence of infinitely many pairs $(\phi,F)$ of even Hecke-Maass cusp form $\phi$ on $\SL_2(\Z)$ and a cohomological Hecke eigen cusp form $F$ on $\SL_3(\Z)$ with $L\left(1/2,\phi\right)L\left(1/2, \phi\times F\right)\not=0$. This should be compared with \cite[Corollary 1]{XLi} and \cite[Corollary 3]{Reznikov}, where given a scalar valued Hecke eigen cusp form $F$ on $\SL_3(\Z)$, the infinitum of even Hecke-Maass cusp form $\phi$ on $\SL_2(\Z)$ with $L\left(1/2,\phi\right)L\left(1/2, \phi \times F\right)\not=0$ is proved by a different technique. In the case when $\phi \times F$ is a cohomological cusp form on $\GL_n \times \GL_{n+1}$ whose conductor is a large power of a given prime number $p$, Januszewski (\cite{Janus}) showed the abundance of the non-vanishing of $L(1/2,\phi\times F)$ by employing a technique from the theory of $p$-adic $L$-functions. 

We are concerned with the family of degree $6$ $L$-functions: 
\begin{align}
L(s,\phi\times {\rm Ad} f),\quad f\in \bigcup_{k \in 2\N\cap [{\rm K},2{\rm K})}H_k,
 \label{FamilyLftn}
\end{align}
where ${\rm K}$ is a positive parameter. For $n\in \N$ and $\phi$ as above, we define the following counting function:
$$
N_{\phi,n}({\rm K})=
\#\biggl\{f \in \bigcup_{{\rm K} \le k<2{\rm K}}H_k \ \biggm| \ \l_{f}(n)\not=0,\,L(1/2,\phi\times \Ad f) \neq 0\biggr \}, \quad {\rm K}>0.  
$$
Let $X(\phi)$ be the set of $n\in \N$ such that $-4n$ is a fundamental discriminant and such that $L(1/2,\phi \otimes \chi_{-4n})\not=0$, where $\chi_{-4n}$ is the Kronecker character of conductor $-4n$. From \cite[Theorem B1]{FriedbergHoffstein}, the set $X(\phi)$ is shown to be an infinite set provided $L(1/2,\phi)\not=0$. 

\begin{thm}\label{QuaNonVan}
Let $\phi$ be an even Hecke-Maass cusp form on $\SL_2(\Z)$ with $L(1/2,\phi)\neq 0$. Then for any $n \in X(\phi)$ and $\eta \in (0,1)$, there exists a constant ${\rm K}(\phi,n, \eta)>0$ such that for any ${\rm K}\ge {\rm K}({\phi,n,\eta})$, we have 
\begin{align}
\frac{N_{\phi,n}({\rm K})}{{\rm K}} \ge \frac{1-\eta}{16\pi}
\,\frac{1}{\sqrt{n}\, d(n)^2}\,\frac{L(1/2, \phi \otimes \chi_{-4n})}{C(\phi)\,L(1,\phi, {\rm sym}^2)},
 \label{QNV-f1}
\end{align}
where 
$$
C(\phi)=\frac{1}{\zeta(2)}\,\prod_{p}\left(1-\frac{p^{-1}{\rm tr}\,A_{p}(\phi)
}{p^{1/2}+p^{-1/2}} \right).
$$
\end{thm}
Note that the both sides of the inequality \eqref{QNV-f1} do not depend on the normalization of $\phi$. The theorem shows that among the family of $L$-functions \eqref{FamilyLftn} the number of non-vanishing central values is at least a positive multiple of ${\rm K}$ asymptotically as ${\rm K}\rightarrow \infty$, whereas the size of the whole family is about ${\rm K}^2$. As mentioned above, the infinitum of non-vanishing central values among the family is a direct consequence of \eqref{LuoSar3}; Theorem~\ref{QuaNonVan} is viewed as a quantification of this fact.

The proof of Theorem~\ref{QuaNonVan} relies not only on the asymptotic formula of ${\bf A}_{k,n}(\phi)$ as $k\rightarrow \infty$ (Theorem~\ref{ASYMPT}) but also on the variance formula \eqref{LuoSar3} in conjunction with the Waldspurger type formulas (\cite{Waldspurger}, \cite{Popa}, \cite{MartinWhitehouse},
\cite{Zhang}, \cite{Watson}, \cite{Ichino}) which interrelate the central values of automorphic $L$-functions with the periods of automorphic forms.

\subsection{Description of main results}
\subsubsection{The trace formula}
In this article, $v$ denotes one of the places of $\Q$, i.e, $v$ is a prime number $p$ expressing a finite place or the symbol $\infty$ expressing the archimedean place.
For a prime number $p$, set $\fX_p=\C/4\pi i (\log p)^{-1}\Z$.
The symbol $\NN$ denotes the set of all the positive integers and set $\NN_0=\NN\cup\{0\}$. For a condition $\rm P$, let us define $\delta({\rm P})$ by $\delta({\rm P})=1$ if $ \rm P$ is true, and $\delta({\rm P})=0$ otherwise. Set $\GL_2(\R)^+=\{g\in \GL_2(\R)|\det(g)>0\}$. 

Let $\phi$ be an even Hecke-Maass cusp form on $\SL_2(\Z)$ as above.
 There exist a complex number $\nu_\infty$ and a family of complex numbers $\nu=\{\nu_p\}_{p<\infty}\in \prod_{p}\fX_p$ such that 
\begin{align}
&-y^2\left(\tfrac{\partial^2}{\partial x^2}+\tfrac{\partial^2}{\partial x^2}\right)\phi (\tau)=\tfrac{1-\nu_\infty^2}{4}\,\phi(\tau), 
\notag
\\
&\phi(p\tau)+\sum_{j=0}^{p-1}\phi\left(\tfrac{\tau+j}{p}\right)=(p^{(\nu_p+1)/2}+p^{(-\nu_p+1)/2})\,\phi(\tau) \quad \text{for all prime numbers $p$}.
\label{pHecke}
\end{align} 
The family $(\nu_\infty,\nu=\{\nu_p\}_{p<\infty})$ is called the spectral parameter of $\phi$. Note that $A_p(\phi)=\diag(p^{\nu_p/2},p^{-\nu_p/2})$ for $p<\infty$. We recall the toric period integrals of Maass forms on $\SL_2(\Z)$ following \cite{Zagier81}. Let $\Dcal$ denote the set of all the fundamental discriminants. For $D\in \Dcal$, let $\chi_{D}:\Z\rightarrow \{0,1,-1\}$ be the Kronecker character of conductor $D$. Let $\Fcal(D)$ be the set of all the primitive binary quadratic forms $Q(x,y)$ with integral coefficients of discriminant $D$ such that it is not negative-definite. The group $\PSL_2(\Z)$ acts on the set $\Fcal(D)$ by the rule 
$$
(Q\cdot \gamma)(x,y)=Q(ax+by,cx+dy), \quad \text{for $\gamma=\left[\begin{smallmatrix} a& b \\ c& d \end{smallmatrix}\right] \in \SL_2(\Z)$}.  
$$
Let $E=\Q(\sqrt{D})$ be the quadratic field of discriminant $D$. Then it is well-known that the cardinality $h=\#(\Fcal(D)/\PSL_2(\Z))$ coincides with the narrow class number of $E$. Fix a complete set of representatives $\{Q_j\}_{j=1}^{h}$ of the $\PSL_2(\Z)$-equivalence classes in $\Fcal(D)$.
We remark that $\Gamma(Q)=\{\gamma\in \PSL_2(\Z)|\,Q \cdot \gamma=Q\,\}$ is a group isomorphic to $U_{D}/(U_{D}\cap\{\pm 1\})$, where $U_{D}$ is the unit group of $E$ if $D<0$
and the totally positive unit group of $E$ if $D>0$, respectively.

If $D<0$, define
\begin{align}
\PP_{D}(\phi)=\frac{2}{w_D}\sum_{j=1}^{h} \phi(z_{Q_j}),
 \label{classicalPeriod1}
\end{align}
where $w_{D}$ denotes the number of roots of unity in $E$, and $z_{Q_j}$ is the root of the quadratic equation $Q_j(z,1)=0$ such that $z_{Q_j}\in \fH$. The right-hand side of \eqref{classicalPeriod1} is independent of the choice of the representatives $\{Q_j\}_{j=1}^h$. Indeed, for $Q\in \Fcal (D)$ and $Q'=Q\cdot \gamma$ with $\gamma \in \SL_2(\Z)$, we have $z_{Q'}=\gamma \langle z_{Q}\rangle$, where $z_Q,z_{Q'}\in \fH$ are roots of $Q(z,1)=0$ and $Q'(z,1)=0$, respectively. Thus $\phi(z_{Q'})=\phi(z_{Q})$ by the modularity of $\phi$.
Suppose $D>0$ and $Q\in \Fcal(D)$.
Then $\Gamma(Q)$ 
is an infinite cyclic group. 
Let $\Omega_{Q}$ be the semicircle on $\fH$ with end points $z_{Q}$ and $z_{Q}'$, the two roots of the quadratic equation $Q(z,1)=0$. Let $g_{Q,\infty}\in \GL_2(\R)$ be a matrix such that $g_{Q,\infty}\langle0\rangle=z_Q$ and $g_{Q,\infty}\langle i\infty\rangle=z_Q'$; by changing the roles of $z_{Q}$ and $z_{Q}'$ if necessary, we may suppose $\det g_{Q,\infty}>0$. By this matrix, we have a parametrization $z=g_{Q,\infty}\langle it\rangle \,(t>0)$ of $\Omega_Q$ by positive real numbers. Let us define a line element $|\d_{Q}z|$ on $\Omega_Q$ as $|\d_{Q}z|=\d t/t$ for $z=g_{Q,\infty}\langle it\rangle \,(t>0)$. The line element thus defined is independent of the choice of a matrix $g_{Q,\infty}\in \GL_2(\R)^+$. Indeed, if $h\in \GL_2(\R)^+$ satisfies $h\langle 0\rangle=g_{Q,\infty}\langle 0 \rangle=z_Q$ and $h\langle i\infty\rangle=g_{Q,\infty}\langle i\infty\rangle =z_Q'$, then the fractional linear transformation by $h^{-1}g_{Q,\infty}$ fixes two points $0$ and $i\infty$ on the Riemannian sphere. Hence there is $\a>0$ such that $h^{-1}g_{Q,\infty}\langle \tau\rangle =\a \tau$ for all $\tau\in \fH$. Then by a change of variables,  
$$\int_{0}^{\infty}f(g_{Q,\infty}\langle it\rangle)\,\tfrac{\d t}{t}=\int_{0}^{\infty}f(h\langle i\a t \rangle )\,\tfrac{\d t}{t}=\int_{0}^{\infty} f(h\langle it \rangle)\, \tfrac{\d t}{t}  
$$
for any integrable function $f$ on $\Omega_{Q}$. Define
\begin{align*}
\PP_{D}(\phi)=\sum_{j=1}^{h}\int_{\Gamma(Q_j)\bsl \Omega_{Q_j}} \phi(z)\,|\d_{Q_j}z|. 
\end{align*}
As in the case of $D<0$, this is shown to be independent of the choice of a set of representatives $\{Q_j\}_{j=1}^h$. 

Let $\Delta$ be a non-zero integer such that $\Delta\equiv 0,1\pmod{4}$; it has a unique decomposition $\Delta=Df^2$ with $f\in \N$ and $D\in \Dcal\cup\{1\}$. For a finite set of prime numbers $S$ and for $\nu=\{\nu_p\}_{p<\infty}\in \prod_{p}\fX_p$, we set
\begin{align}
{\bf B}^{S}(\nu ;\Delta)=\prod_{\substack{p|f \\ p\not\in S}}\left\{\frac{\zeta_p(-\nu_p)}{L_p\left(\frac{-\nu_p+1}{2},\chi_{D} \right)}|f|_p^{\frac{\nu_p-1}{2}}+\frac{\zeta_p(\nu_p)}{L_p\left(\frac{\nu_p+1}{2},\chi_{D} \right)}|f|_p^{\frac{-\nu_p-1}{2}}\right\},
 \label{bfB-defformula}
\end{align}
and ${\bf B}(\nu ;\Delta)={\bf B}^{\emp}(\nu ;\Delta)$, where $L_p(s,\chi_D)$ is the local $p$-factor of the Dirichlet $L$-function of $\chi_D$ and $|f|_p$ is the normalized $p$-adic absolute value of $f$. 

For $z\in \C$ and $a\in \R$, define
\begin{align*}
\Ocal_k^{+,(z)}(a)&=\frac{2\pi}{\Gamma(k)}\frac{\Gamma\left(k+\tfrac{z-1}{2} \right) \Gamma\left(k+\tfrac{-z-1}{2} \right)}{\Gamma_\R\left(\tfrac{1+z}{2}\right)\Gamma_\R\left(\tfrac{1-z}{2}\right)}\delta(|a|>1)\,(a^{2}-1)^{1/2}\fP^{1-k}_{\frac{z-1}{2}}(|a|), \quad \\
\Ocal_k^{-,(z)}(a)&=\frac{\pi i }{\Gamma(k)} \Gamma\left(k+\tfrac{z-1}{2} \right) \Gamma\left(k+\tfrac{-z-1}{2} \right)\,{\rm sgn}(a)\, (1+a^{2})^{1/2}\{\fP^{1-k}_{\frac{z-1}{2}}(ia)-\fP^{1-k}_{\frac{z-1}{2}}(-ia)\},
\end{align*}
where $\fP_{\nu}^{\mu}(x)$ is the associated Legendre function of the 1st kind which is defined for points $x\in \C$ outside the interval $(-\infty,+1]$ of the real axis (\cite[\S 4.1]{MOS}).
Note that $\Ocal_k^{-,(z)}(0)$ is understood as $\lim_{a\rightarrow 0}\Ocal_k^{-,(z)}(a)$.

Now we state the trace formula, which will be proved in \S \ref{Proof of Theorem MVT-L3-1}.
\begin{thm} \label{MVT-L3-1}
Let $n$ be a positive integer. For any $D\in \Dcal$, let $\cT(n,D)$ be the set of $t\in \Z$ such that $t^2-4n=f^2D$ with some $f\in \N$. Then 
{\allowdisplaybreaks
\begin{align}
\tfrac{4\pi}{k-1}n^{1/2} \,{\bf A}_{k,n}(\phi)
=&\tfrac{1}{4}\hat L\left(\tfrac{1}{2},\phi\right)\sum_{\substack{n=d_1d_2 \\ d_1,d_2>0,d_1\not=d_2}}
{\bf B}(\nu ;(d_1-d_2)^2)
\,\Ocal_{k}^{+,(\nu_\infty)}\left(\tfrac{d_1+d_2}{d_1-d_2}\right)
 \label{MVT-L3-1-f1} 
\\
&+\tfrac{1}{2} \sum_{D\in \Dcal} 2^{\delta(D<0)}\PP_{D}(\phi)\sum_{t\in \Tcal(n,D)}{\bf B}(\nu ;t^2-4{n})\,\Ocal_k^{\sgn(t^2-4{n}),(\nu_\infty)}\left(\tfrac{t}{\sqrt{|t^2-4n|}}\right),
 \notag
\end{align}}where $\hat L(s,\phi)=\Gamma_{\R}(s+\nu_\infty/2)\Gamma_{\R}(s-\nu_\infty/2)L(s,\phi)$ is the completed $L$-function of $\phi$.
The summation over $(D,t)$ on the right-hand side of \eqref{MVT-L3-1-f1} converges absolutely.  
\end{thm}
Let 
$$E(z,\tau)=\sum_{\gamma \in \{\pm\left[\begin{smallmatrix} 1 & * \\ 0 & 1 \end{smallmatrix}\right]\}\bsl \SL_2(\Z)} \Im(\gamma \tau)^{(z+1)/2}, \quad \Re z>1, \,\tau \in \fH$$
 be the non-holomorphic Eisenstein series on $\SL_2(\Z)$ and set $E^{*}(z)=\hat \zeta(z+1)E(z)$, where $\hat\zeta(s)$ denotes the completed Riemann zeta function.
Then for any $D\in \Dcal$,
\begin{align}\label{ZagierEisen}
\PP_{D}(E^*(z))=2^{-\delta(D<0)}|D|^{(z+1)/4}\hat \zeta_{\Q(\sqrt{D})}((z+1)/2), \quad \Re z>1
\end{align}
holds from \cite[\S2, Examples 1 and 2]{Zagier81} (cf.\ \cite[Proposition 7.7]{SugiyamaTsuzuki2018}), where $\hat\zeta_{\Q(\sqrt{D})}(s)$ denotes the completed Dedekind zeta function of $\Q(\sqrt{D})$.
Moreover, we have
$\hat L(1/2,E^{*}(z))=\hat \zeta\left(\tfrac{z+1}{2}\right)\hat\zeta\left(\tfrac{1-z}{2}\right)$
and
$$\mu_f(E^{*}(z))=\frac{\hat\zeta(\frac{z+1}{2})\hat L(\frac{z+1}{2}, f, {\rm sym}^2)}{2\,\hat L(1,f,{\rm sym}^2)}.$$
In the right-hand side, $\hat L(\frac{z+1}{2}, f, {\rm sym}^2)$ is the completed symmetric square $L$-function attached to $f$.
Thus, the formula in \cite[Theorem 1]{Zagier} can be equivalently written in the following form for any $n\in \N$ and $3-2k<\Re z<2k-3$: 
{\allowdisplaybreaks\begin{align}\label{Zagierformula}
&\tfrac{4\pi}{k-1}n^{1/2}\sum_{f \in H_k} \mu_f(E^{*}(z)) \,\l_{f}(n) \\
=& \delta(\sqrt{n}\in \N)\,\hat\zeta\left(\tfrac{z+1}{2}\right)n^{1/2} 
\biggl\{\hat\zeta(-z)\, 2^{1-z}\pi^{(3-z)/4}
\frac{\Gamma(k+(z-1)/2)}{\Gamma(k)\Gamma((z+1)/4)}\,n^{-(z+1)/4} \notag \\
&
+\hat\zeta(z)\, 2^{1+z}\pi^{(3+z)/4}\frac{\Gamma(k+(-z-1)/2)}{\Gamma(k)\Gamma((-z+1)/4)}\,n^{-(-z+1)/4}
\biggr\} \notag
\\
&+\tfrac{1}{2}\hat L(1/2,E^*(z))
\sum_{\substack{n=d_1d_2 \\ d_1,d_2>0,d_1\not=d_2}}{\bf B}({\underline z} ;(d_1-d_2)^2)
\,\Ocal_{k}^{+,(z)}\left(\tfrac{d_1+d_2}{d_1-d_2}\right) \notag
\\
&+\tfrac{1}{2} \sum_{D\in \Dcal} 2^{\delta(D<0)}\PP_{D}(E^{*}(z))\sum_{t\in \Tcal(n,D)}{\bf B}({\underline z};t^2-4{n})\,\Ocal_k^{\sgn(t^2-4{n}),(z)}\left(\tfrac{t}{\sqrt{|t^2-4n|}}\right),\notag
\end{align}
}where ${\underline z}$ is the diagonal image of $z\in \C$ in $\prod_{p}\fX_p$.
Theorem \ref{MVT-L3-1} can be regarded as an
analogue of this formula for Maass cusp forms.
 For deduction of the above formula from \cite[Theorem 1]{Zagier}, we use the relation
$$
\sum_{0<d|f}\mu(d)\left(\tfrac{D}{d}\right)\,d^{-(z+1)/2}\sigma_{-z}(f/d)
=f^{-(z+1)/2}{\bold B}({\underline z};\Delta),
$$
which is proved by an elementary computation for any non-zero discriminant $\Delta=f^2D$. Here $\mu(d)$ is the M\"{o}bius function and $\sigma_{-z}(n)=\sum_{0<d|n}d^{-z}$ is the divisor function.

\subsubsection{Asymptotic formula}

For convenience, we set $\PP_{\Delta}(\phi)=\PP_{D}(\phi)$ if $\Delta=f^2D$ with $f\in \N$ and $D\in \Dcal$. One of the main ingredients of the proof of Theorem~\ref{QuaNonVan} is the following asymptotic formula of ${\bf A}_{k,n}(\phi)$ deduced from Theorem~\ref{MVT-L3-1}.

\begin{thm} \label{ASYMPT} For any positive integer $n$,
\begin{align*}
 &(-1)^{k/2}n^{1/2}{\bf A}_{k,n}(\phi) \\
= & \tfrac{1}{2}\PP_{-4n}(\phi){\bf B}(\nu;-4n) \\
&+\tfrac{1}{2}(-1)^{k/2}\sum_{\substack{t \in\ZZ,\\ 0<|t|<2\sqrt{n}}}
\PP_{t^2-4n}(\phi)\bfB(\nu; t^2-4n)
(i\sqrt{|\Delta|})^{-1}\,
\left\{\rho\,(\bar \rho/\rho)^{k/2}
-\bar \rho\,(\rho/\bar \rho)^{k/2}\right\}  +O(k^{-1})
\end{align*}
as $k \rightarrow \infty$. Here $\Delta=t^2-4n$ and $\rho=2^{-1}(-t+i\sqrt{|\Delta|})$.
\end{thm}

As mentioned above, Theorem \ref{ASYMPT} can be regarded as a refinement of
\cite[Theorem in p.22]{Luo} of the first moment of $\mu_f(\phi)$. The second main term on the right-hand side is bounded with an oscillatory behavior as $k\rightarrow \infty$. By taking an average with respect to the weight $k$ on the interval $[{\rm K},2{\rm K})$, we can make the oscillatory term smaller to have the following: 	\begin{thm}\label{ASYMPT_average}
	For any positive integer $n$, 
	\begin{align*}
	\lim_{{\rm K}\rightarrow \infty} \frac{1}{2^{-1}{\rm K}}\sum_{k \in 
2\N \cap [{\rm K}, 2{\rm K})}
	(-1)^{k/2}\bfA_{k,n}(\phi)=\tfrac{1}{\sqrt{4n}}\PP_{-4n}(\phi){\bf B}(\nu;-4n).
	\end{align*}
\end{thm}
The proof of these theorems will be given in \S 3.1. 

%


\subsection{Organization}
In \S2 we first recall necessary materials from \cite[\S9]{SugiyamaTsuzuki2018} in our setting and then give a proof of Theorem~\ref{MVT-L3-1} translating the adelic language into the classical one. There are two ways to define the notion of toric period integrals of cusp forms, one for the cusp forms over adeles as in \cite{Popa} and \cite{SugiyamaTsuzuki2018} and the other for classical cusp forms on the upper-half plane as in \cite{Zagier81} and \cite{KatokSarnak}. In \S2.1, we make an explicit relation between the two definitions. In \S2.2, we recall the main result of \cite[\S10]{SugiyamaTsuzuki2018}. The final subsection \S2.3 contains a detailed argument how Theorem~\ref{MVT-L3-1} is deduced from the Jacquet-Zagier type trace formula on adeles recalled in \S2.2. In \S3 we show the asymptotic formulas in Theorems~\ref{ASYMPT} and \ref{ASYMPT_average}. In \S4, we prove Theorem~\ref{QuaNonVan} invoking Theorem~\ref{ASYMPT_average} and Theorem~\ref{thmLuoSar}. The final section \S5 is independent of the other parts of the article and should be regarded as a complement of \cite[\S 10]{SugiyamaTsuzuki2018}; we complete \cite[Theorem 7.9 (5)]{SugiyamaTsuzuki2018} by computing the orbital integrals on $\GL_2(\Q_2)$ to give its explicit formula, which allows us to lift the assumption in \cite[Theorem 1.1]{SugiyamaTsuzuki2018} that $S$ should be disjoint from the dyadic places.

\section{Proof of Theorem~\ref{MVT-L3-1}}\label{Proof of Theorem MVT-L3-1}
Let $\psi=\prod_{v}\psi_v$ be a unitary character of $\Q\bsl \A_{\Q}$ such that $\psi_\infty(x)=e^{2\pi i x}$ for $x\in \R$.
Recall that $\phi$ is an even Hecke-Maass cusp form on $\SL_2(\Z)$ with spectral parameter $(\nu_\infty,\nu=\{\nu_{p}\}_{p<\infty})$. By the decomposition $\GL_2(\A_\Q)=\GL_2(\Q)\GL_2(\R)^{+}\GL_2(\hat \Z)$, we lift $\phi$ to a function $\tilde \phi$ on $\GL_2(\A_\Q)$ such that 
$$\tilde \phi(\gamma g h)=\phi\left(\tfrac{ai+b}{ci+d}\right) \quad \text{for $\gamma\in \GL_2(\Q)$, $g=\left[\begin{smallmatrix} a & b \\ c & d \end{smallmatrix}\right]\in \GL_2(\R)^{+}$ and $h \in \GL_2(\hat \Z)$}.$$
Then the right translations of $\tilde \phi$ by elements of $\GL_2(\A_\Q)$ span an irreducible cuspidal automorphic representation $\pi=\pi_\phi$ of $\GL_2(\A_\Q)$ isomorphic to the restricted tensor product 
\begin{align}
\bigotimes_{v}\cW_{\psi_v}(\nu_v) \cong \pi_{\phi},
 \label{SymBreak}
\end{align}
where $\cW_{\psi_v}(\nu_v)$ is the Whittaker model of the principal series representation ${\rm Ind}_{B(\Q_v)}^{\GL_2(\Q_v)}(|\cdot|_v^{\nu_v/2}\boxtimes |\cdot|_v^{-\nu_v/2})$ with respect to the character $\left[\begin{smallmatrix} 1 & x \\ 0 & 1 \end{smallmatrix}\right] \mapsto \psi_{v}(x)$ (see for example \cite[\S 3.6]{Bump}). The isomorphism \eqref{SymBreak} is given by the map which sends $W \in \bigotimes_{v}\cW_{\psi_v}(\nu_v)$ (viewed as a global Whittaker function on $\GL_2(\A_\Q)$) to the cusp form $g \mapsto \sum_{\alpha \in \Q^\times} W\left(\left[\begin{smallmatrix} \alpha & 0 \\ 0 & 1 \end{smallmatrix}\right]g\right)$ in $\pi_{\phi}$. Under the isomorphism \eqref{SymBreak} the function $\tilde \phi$ corresponds to the pure tensor $W_{\pi}^{\rm new}=2^{-1}\cdot\otimes_{v}W_v^{(\nu_v)}$ with $W_v^{(\nu_v)}$ the local new vector of $\cW_{\psi_v}(\nu_v)$ in the sence that the local zeta-integral $\int_{\Q_v^\times}W_{v}^{(\nu_v)}\left(\left[\begin{smallmatrix} t & 0 \\ 0 & 1 \end{smallmatrix}\right]\right)\,|t|_{v}^{s-1/2}\d^\times t_v$ coincides with the local $v$-factor of $\hat L(s,{\phi})$ when $\Re s \gg 1$. The constant $2^{-1}$ in $W_{\pi}^{\rm new}$ arises from our normalization of $\phi$ by \eqref{phiNormal} and the formula $W_\infty^{(\nu_\infty)}\left(\left[\begin{smallmatrix} y & 0 \\ 0 & 1 \end{smallmatrix}\right]\right)=2|y|^{1/2}K_{\nu_\infty/2}(2\pi |y|)$ for $y\in \R^\times$.

It is worth noting that $A_p(\phi)=\diag(p^{\nu_p/2},p^{-\nu_p/2})$ for all $p<\infty$. From the unitarity of $\pi_{\phi}$, we have $\Re \nu_p=0$ or $\Im \nu_p\in \{0, 2\pi (\log p)^{-1}\}$, $|\Re \nu_p|<1$ for all $p$. 

\medskip
\noindent
{\bf Remark} : 
From the Fourier expansion
\begin{align*}
	E^*(z,\tau)=\hat\zeta(-z)y^{(1+z)/2}+\hat\zeta(z)y^{(1-z)/2}+
	\sum_{n\in \Z-\{0\}}|n|^{z/2}\sigma_{-z}(|n|)\,2y^{1/2}K_{z/2}(2\pi|n|y)e^{2\pi i nx}
\end{align*}
for $\tau=x+iy \in \fH$,
it is seen that the lift $\tilde \phi$ of $\phi=E^{*}(z)$ precisely corresponds to the new vector $W_{\pi}^{\rm new}=\otimes_{v}W_v^{(z)}$
in the automorphic representation $\pi$ generated by $\tilde\phi$. This explains that the coefficient $1/4$ of $\hat L(1/2,\phi)$ in \eqref{MVT-L3-1-f1} is replaced with $1/2$ in the Eisenstein analogue.

\subsection{Period integrals} \label{sect:PeriodInt}
Let $\Delta=Df^2$ with $D\in \Dcal$ and $f\in \N$ be as before and $E=\Q(\sqrt{D})$. Define a $\Q$-algebra embedding $\iota_{\Delta}:E\hookrightarrow \Mat_2(\Q)$ by \begin{align*}
\iota_{\Delta}(a+b\sqrt{4^{-1}\Delta})&=\left[\begin{smallmatrix} a & b \\ 4^{-1}\Delta b & a \end{smallmatrix}\right], \quad a,b\in \Q.
\end{align*}
For a place $v$ of $\Q$, let $\d \lambda_v$ be a Haar measure on the additive group $\Q_v$ such that $\lambda_p(\Z_p)=1$ if $v=p<\infty$ and $\lambda_\infty([0,1])=1$ if $v=\infty$. Then we fix Haar measures on the multiplicative groups $\A_\Q^\times$ and $\A_E^\times$ by $\d^\times t=\otimes_{v} \zeta_{v}(1)\frac{\d \lambda_v(t_v)}{|t_v|_v}$ and
$\d^\times \tau=\otimes_{v} \zeta_{E_v}(1) \frac{\d\lambda_v(x_v)\d \lambda_v(y_v)}{|x_v^2-4^{-1}\Delta y_v^2|_v}$, respectively,
and then endow $\A_\Q^\times \bsl \A_E^\times$ with the quotient measure $\d^\times \dot\tau=\d^\times \tau/\d^\times t$. Here $\zeta_v(s)$ (resp.\ $\zeta_{E_v}(s)$) is the local $v$-factor of the completed Dedekind zeta function $\hat\zeta(s)$ (resp.\ $\hat\zeta_E(s)$) of $\Q$ (resp.\ $E$). Fix $m_p\in \QQ_p^\times$ such that $4^{-1}\Delta m_p^{-2}\in p\Z_p\cup\{1\}\cup\{\Z_p^\times-(\ZZ_p^\times)^2\}$.
Let $R_{\Delta}=(R_{\Delta,v})\in \GL_2(\A_\Q)$ be an element defined as
\begin{align*}
&R_{\Delta,p}=\begin{cases}
\left[\begin{smallmatrix} m_p & 0 \\ 0 & 1 \end{smallmatrix}\right] \quad &\text{(if $\chi_{D}(p)=0,-1$)}, \\
\left[\begin{smallmatrix} 1 & 1 \\ 1 & -1 \end{smallmatrix}\right]\left[\begin{smallmatrix} m_p & 0 \\ 0 & 1 \end{smallmatrix}\right] &\text{(if $\chi_D(p)=1$)}, 
\end{cases} \quad
&R_{\Delta,\infty}= \begin{cases}
\left[\begin{smallmatrix} \sqrt{|4^{-1}\Delta|} & 0 \\ 0 & 1 \end{smallmatrix}\right] &\text{(if $D<0$)}, \\
\left[\begin{smallmatrix} 1 & 1 \\ 1 & -1 \end{smallmatrix}\right]\left[\begin{smallmatrix} \sqrt{4^{-1}\Delta} & 0 \\ 0 & 1 \end{smallmatrix}\right]
 &\text{(if $D>0$)}
\end{cases}
\end{align*}
(cf.\ \cite[\S7]{SugiyamaTsuzuki2018}). From \cite[\S 9]{SugiyamaTsuzuki2018}, recall that the period integral of $\tilde \phi$ is defined to be 
\begin{align*}
\Pcal_{D}(\tilde \phi) = |D|^{1/2}\int_{\A_{\Q}^\times E^\times \bsl \A_{E}^\times}\tilde \phi(\iota_{\Delta}(\tau)R_{\Delta}^{-1})\,\d^\times \dot\tau \times \begin{cases} 1 & (\text{if $\chi_{D}(2)=0,1$}), \\
2^{-(1+\nu_2)/2}3(1+2^{-\nu_2})^{-1} & (\text{if $\chi_{D}(2)=-1$}).
\end{cases}
\end{align*}

\begin{prop} \label{PeriodL0} 
 Let $\phi$ be a Hecke-Maass form on $\SL_2(\Z)$ and $\tilde \phi$ its lift to the adelization $\GL_2(\A_\Q)$. Then for any $\Delta=f^2D$ with $D\in \Dcal$ and $f\in \N$, we have the equality $\Pcal_{\Delta}(\tilde \phi)=2^{\delta(D<0)}\PP_{D}(\phi)$. 
\end{prop}
 Here is some detail of the arguments. Basically we follow \cite[\S6]{Popa}.
 Define a $\Q$-algebra embedding $\Psi_{D}:E \rightarrow {\rm M}_2(\Q)$ by 
$$
\Psi_{D}(\sqrt{D})=\left[\begin{smallmatrix} 0 & 2 \\ D/2 & 0 \end{smallmatrix}\right] \quad \text{ if $D\equiv 0 \pmod{4}$}
$$
 and by 
$$
\Psi_{D}(\sqrt{D})=\left[\begin{smallmatrix} (D+1)/2 & (D-1)/2 \\ -(D-1)/2 & -(D+1)/2 \end{smallmatrix} \right]=\left[\begin{smallmatrix} 1 & 1 \\ 1 & -1 \end{smallmatrix}\right]\left[\begin{smallmatrix} 0 & 1 \\ D & 0 \end{smallmatrix}\right]\left[\begin{smallmatrix} 1 & 1 \\ 1 & -1 \end{smallmatrix}\right]^{-1}
 \quad \text{if $D\equiv 1\pmod{4}$}. 
$$
Then it turns out that the embedding $\Psi_{D}$ is optimal of level $1$, i.e., $\Psi_D(E)\cap {\rm M}_2(\Z)=\Psi_D(\cO_{E})$ (\cite[\S 6.2]{Popa}).
For a rational matrix $\delta\in \GL_2(\Q)$ and a place $v$ of $\Q$, let $\delta_{v}$ denote the image of $\delta$ in $\GL_2(\Q_v)$. We write $\delta_\Q$ for the diagonal image of $\delta$ in $\GL_2(\A_\Q)$. Set $\delta_\fin=(\delta_p)_{p} \in \GL_2(\AA_{\Q, \fin})$.

\begin{lem}
\label{Period0712}	 Set $\Delta=f^2D$ with $f\in \N$. Then,
\begin{align*}
\Pcal_{\Delta}(\tilde \phi)
& = |D|^{1/2} \int_{\A_\Q^\times E^\times\bsl \A_E^\times} \tilde\phi(\Psi_{D}(\tau)\gamma_{D,2}\gamma_{D,\infty})\d^\times \dot \tau
\times\begin{cases} 1, \quad & (\chi_{D}(2)= 0,1), \\ 3(2^{(\nu_2+1)/2}+2^{(-\nu_2+1)/2})^{-1}, \quad & (\chi_{D}(2)=-1), \end{cases}
\end{align*}
where 
\begin{align*}
\gamma_{D,2}&=
\begin{cases} \left[\begin{smallmatrix} 1& 1 \\ 1& -1 \end{smallmatrix} \right]_2 & (\chi_{D}(2)=-1), \\ 
(1_2)_{2} & (\chi_{D}(2)=0, 1),
\end{cases} \\ 
\gamma_{D,\infty}&=\begin{cases}
\left[\begin{smallmatrix} \sqrt{|D|}/2& 0  \\ 0 &1 \end{smallmatrix} \right]_\infty^{-1} \quad (D\equiv 0 \pmod{4}), \\
\left[\begin{smallmatrix} 1& 1 \\ 1& -1 \end{smallmatrix} \right]_\infty
\left[\begin{smallmatrix} \sqrt{|D|} & 0  \\ 0 & 1 \end{smallmatrix} \right]_\infty^{-1}\left[\begin{smallmatrix} 1& 1 \\ 1 &-1 \end{smallmatrix} \right]_\infty^{-1} \quad (D\equiv 1\pmod{4}). 
\end{cases}
\end{align*}
\end{lem}
\begin{proof}
Suppose $D\equiv 1\pmod{4}$. Then $D$ is square-free and hence we may take $m_p=f/2$. We have 
$$
\Psi_{D}=\Ad\left(\left[\begin{smallmatrix} 1& 1 \\ 1& -1 \end{smallmatrix} \right]\left[\begin{smallmatrix} f/2 & 0  \\ 0 & 1 \end{smallmatrix} \right]\right) \circ \iota_{\Delta}
$$
and
$$
R_{\Delta,p}=\begin{cases}
\left[\begin{smallmatrix} f/2 & 0  \\ 0 & 1 \end{smallmatrix} \right]_p \quad (\chi_{D}(p)=0,-1), \\
\left[\begin{smallmatrix} 1& 1 \\ 1& -1 \end{smallmatrix} \right]_p\left[\begin{smallmatrix} f/2 & 0 \\ 0 & 1 \end{smallmatrix} \right]_p \quad (\chi_{D}(p)=1),
\end{cases}
\qquad 
R_{\Delta,\infty}=\begin{cases}
\left[\begin{smallmatrix} f\sqrt{|D|}/2 & 0 \\ 0 & 1 \end{smallmatrix} \right]_\infty \quad (D<0), \\
\left[\begin{smallmatrix} 1 &1 \\ 1 &-1 \end{smallmatrix} \right]_\infty
\left[\begin{smallmatrix} f\sqrt{D}/2 & 0 \\ 0& 1 \end{smallmatrix} \right]_\infty \quad (D>0).
\end{cases}
$$
By these, 
\begin{align*}
\iota_{\Delta}(\tau)R_{\Delta}^{-1}
&=\left[\begin{smallmatrix} f/2 & 0 \\ 0 & 1 \end{smallmatrix} \right]_\Q^{-1}\left[\begin{smallmatrix} 1 &1 \\ 1 &-1 \end{smallmatrix} \right]_\Q^{-1}\, \Psi_{D}(\tau) \,
\left[\begin{smallmatrix} 1 &1 \\ 1 &-1 \end{smallmatrix} \right]_\Q \left[\begin{smallmatrix} f/2 & 0 \\ 0 & 1 \end{smallmatrix} \right]_\Q
\times \prod_{p<\infty} R_{\Delta,p}^{-1}\times R_{\Delta,\infty}^{-1}
\\
&=\gamma_{\Q} \Psi_D(\tau) \,g_{\fin}\,g_{\infty}
\end{align*}
with $\gamma_{\Q}=\left[\begin{smallmatrix} f/2 & 0 \\ 0 & 1 \end{smallmatrix} \right]_\Q^{-1}\left[\begin{smallmatrix} 1 &1 \\ 1 &-1 \end{smallmatrix} \right]_\Q^{-1}$ and 
\begin{align*}
g_{\fin}&:=\prod_{p<\infty} \left[\begin{smallmatrix} 1 &1 \\ 1 &-1 \end{smallmatrix} \right]_p \left[\begin{smallmatrix} f/2 & 0 \\ 0 & 1 \end{smallmatrix} \right]_p R_{\Delta,p}^{-1}, \qquad
g_{\infty}:=\left[\begin{smallmatrix} 1 &1 \\ 1 &-1 \end{smallmatrix} \right]_\infty \left[\begin{smallmatrix} f/2 & 0 \\ 0 & 1 \end{smallmatrix} \right]_\infty R_{\Delta,\infty}^{-1}.
\end{align*}
Since $\left[\begin{smallmatrix} 1 & 1 \\ 1 & -1 \end{smallmatrix}\right]_p \in \GL_2(\Z_p)$ if $p\not=2$ and $\chi_{D}(2)\in \{1,-1\}$, we have $g_{\fin} \in \gamma_{D,2}\GL_2(\hat \Z)$.
Let $Z$ be the center of $\GL_2$.
By 
$$
\left[\begin{smallmatrix} 1 & 1 \\ 1 & -1 \end{smallmatrix}\right]_\infty
=\sqrt{2} \left[\begin{smallmatrix} \cos(\pi/4) & -\sin(\pi/4) \\ \sin(\pi/4) & \cos(\pi/4) \end{smallmatrix}\right]_\infty \left[\begin{smallmatrix} 1 & 0 \\ 0 & -1 \end{smallmatrix}\right]_\infty\in Z(\R){\rm O}(2), 
$$
we have $g_\infty\in \gamma_{D,\infty}Z(\R){\rm O}(2)$. By the $(Z(\A_\Q)\GL_2(\Q),\bK)$-invariance of $\tilde \phi$, we have the conclusion. The case $D\equiv 0\pmod{4}$ is similar since $\iota_{\Delta} = \Ad([\begin{smallmatrix} 1 & 0 \\ 0 & f \end{smallmatrix}])\Psi_D$ and we may take $m_p=f$.
\end{proof}

\begin{lem} \label{PeriodL1}
Suppose $D\equiv 1\pmod {4}$ and $\chi_{D}(2)=-1$. Then 
\begin{align*}
\int_{u\in \cO_{E_2}^\times} \tilde \phi(g \Psi_{D}(u)\gamma_{D,2})\,\d^\times u&=\tfrac{2}{3} \sum_{\eta \in T(2)/\GL_2(\Z_2)}\tilde \phi(g \eta)
\\
&=\tfrac{2}{3}(2^{(\nu_2+1)/2}+2^{(-\nu_2+1)/2})\tilde \phi(g)
, \quad g\in \GL_2(\A_\Q),
\end{align*}
where $T(2)=\{h\in \Mat_2(\Z_2) \mid |\det h|_2=2^{-1}\}$. 
\end{lem}
\begin{proof} From definition, we have
$$
\Psi_D(a+b\sqrt{D})\gamma_{D,2}=\gamma_{D,2}\left[\begin{smallmatrix} a & b \\ Db & a \end{smallmatrix} \right]
$$
for $a,b\in \Q_2$. Since $D\equiv 1 \pmod{4}$ and $\chi_{D}(2)=-1$, we can set $D=5t^{2}$ with $t\in \Z_2^\times$. Thus
\begin{align*}
\Psi_D(a+bt^{-1}\sqrt{D})\gamma_{D,2}=\gamma_{D,2}\left[\begin{smallmatrix} t^{-1} & 0 \\ 0 & 1 \end{smallmatrix} \right]
\left[\begin{smallmatrix} a & b \\ 5b & a \end{smallmatrix} \right] 
\left[\begin{smallmatrix} t &0 \\ 0& 1 \end{smallmatrix} \right]
\end{align*}
and $\cO_{E_2}^\times=\{u=x+\sqrt{5}y|\,x,y\in 2^{-1}\Z_2,\,x-y\in \Z_2,\,x^2-5y^2\in \Z_2^\times\}$. From this, 
\begin{align}
&\int_{u\in \cO_{E_2}^\times} \tilde \phi(g \Psi_D(u)\gamma_{D,2})\,\d^\times u
 \notag
\\
= & \int_{\substack{(x,y)\in (2^{-1}\Z_2)^2 \\ x-y\in \Z_2,\,x^2-5y^2\in \Z_2^\times}} \tilde \phi(g \gamma_{D,2}\left[\begin{smallmatrix} t^{-1} & 0 \\ 0 & 1 \end{smallmatrix} \right]
\left[\begin{smallmatrix} x & y \\ 5y & x \end{smallmatrix} \right] 
)\,\zeta_{E_2}(1)\,\d x \,\d y. \label{2017/12/17-1}
 \end{align}
Set $T=\{
\left[\begin{smallmatrix} x & y \\ 5y & x \end{smallmatrix} \right] \in \GL_2(\Q_2)|\,\tau=x+\sqrt{5}y\in \cO_{E_2}^\times\}$ and $T^{+}=\GL_2(\Z_2)\cap T$. From \cite[Lemma 7.2]{SugiyamaTsuzuki2018}, we have $T^{+}=\{
\left[\begin{smallmatrix} x & y \\ 5y & x \end{smallmatrix} \right]|x,y\in \Z_2,x^2-5y^2\in \Z_2^\times\}$, 
\begin{align}
T=T^{+}\sqcup \left[\begin{smallmatrix} 1 & 1 \\ 5 & 1 \end{smallmatrix}\right]T^{+}\sqcup \left[\begin{smallmatrix} 3 & 1 \\ 5 & 3 \end{smallmatrix}\right] T^{+}.
 \label{2017/12/17-2}
\end{align}
By the right $\GL_2(\Z_2)$-invariance of $\tilde \phi$, the integral \eqref{2017/12/17-1} becomes the product of 
\begin{align*}
\zeta_{E_2}(1)\int_{\substack{(x,y)\in \Z_2^2 \\ x-y\in \Z_2^\times}}\d x\,\d y=(1-2^{-2})^{-1}\{1\times 1-(\tfrac{1}{2})^2-(\tfrac{1}{2})^2\}=\tfrac{2}{3}
\end{align*}
and
\begin{align*}
&\tilde \phi(g \gamma_{D,2}\left[\begin{smallmatrix} t^{-1} & 0 \\ 0 & 1 \end{smallmatrix} \right])
+\tilde \phi(g \gamma_{D,2}\left[\begin{smallmatrix} t^{-1} & 0 \\ 0 & 1 \end{smallmatrix} \right]\left[\begin{smallmatrix} 1 & 1 \\ 5 & 1 \end{smallmatrix}\right])
+\tilde \phi(g \gamma_{D,2}\left[\begin{smallmatrix} t^{-1} & 0 \\ 0 & 1 \end{smallmatrix} \right]\left[\begin{smallmatrix} 3 & 1 \\ 5 & 3 \end{smallmatrix}\right])
\\
=&
\tilde \phi(g \gamma_{D,2})
+\tilde \phi(g \gamma_{D,2}\left[\begin{smallmatrix} t^{-1} & 0 \\ 0 & 1 \end{smallmatrix} \right]\left[\begin{smallmatrix} 1 & 1 \\ 5 & 1 \end{smallmatrix}\right]
)
+\tilde \phi(g \gamma_{D,2}\left[\begin{smallmatrix} t^{-1} & 0 \\ 0 & 1 \end{smallmatrix} \right]\left[\begin{smallmatrix} 3 & 1 \\ 5 & 3 \end{smallmatrix}\right]
)
\end{align*}
The following three elements belongs to the set $T(2)$: 
$$
\gamma_{D,2}, \quad \tfrac{1}{2}\gamma_{D,2}\left[\begin{smallmatrix} t^{-1} & 0 \\ 0 & 1 \end{smallmatrix} \right]\left[\begin{smallmatrix} 1 & 1 \\ 5 & 1 \end{smallmatrix}\right]
\quad 
\tfrac{1}{2}\gamma_{D,2}\left[\begin{smallmatrix} t^{-1} & 0 \\ 0 & 1 \end{smallmatrix} \right]\left[\begin{smallmatrix} 3 & 1 \\ 5 & 3 \end{smallmatrix}\right]
$$
From the disjoint decomposition \eqref{2017/12/17-2}, these elements belong to  different right $\GL_2(\Z_2)$-orbits. Since $\#(T(2)/\GL_2(\Z_2))=3$, these elements actually form a complete set of representatives of the orbit space $T(2)/\GL_2(\Z_2)$. This completes the proof of the first equality. The second equality follows from the fact that $\tilde \phi$ is a Hecke eigenform with eigenvalue $(2^{(1+\nu_2)/2}+2^{(1-\nu_2)/2})$ at the place $2$. 
\end{proof}
Recall that the measure $\d^\times \tau=\otimes_{v}\d^\times \tau_v$ on $\A_E^\times$ is given as $\d^\times \tau=\otimes_{v}\zeta_{E_v}(1) \frac{\d a_v \d b_v}{|a_v^2-(\Delta/4)b_v^2|_v}$ for $\tau_v = a_v+\sqrt{4^{-1}\Delta} \,b_v\,(a_v,b_v\in \Q_v)$. 

\begin{lem} \label{PeriodL3} 
We have $\vol(\widehat\cO_{E}^\times;\d^\times \tau_\fin)=|f/2|_\infty \times 2$.
\end{lem}
\begin{proof} 
Since $\Delta=f^2D$, by a change of variables and the product formula $\prod_{v}|f/2|_v=1$, we have $\d^\times \tau_{\fin} = |f/2|_\infty \times \otimes_{p} \zeta_{E_p}(1) \frac{\d x_p\d y_p}{|x^2_p-Dy_p^2|_p}$ if $\tau=x+\sqrt{D}y$. We have $\cO_{E_p}=\Z_p+\sqrt{D}\Z_p$ and $\vol(\cO_{E_p}^\times)=1$ unless $p=2$ and $\chi_D(2)\in \{0,-1\}$. If $p=2$ and $\chi_D(2)=0$, then $\go_{E_p} =\Z_2+\sqrt{4^{-1}D}\Z_2$ and $\vol(\go_{E_p}^\times)=\int_{x, 2y \in \Z_2, x^2-Dy^2\in \Z_2^\times}(1-2^{-1})^{-1}dxdy=2$. If $p=2$ and $\chi_{D}(2)=-1$, then $\go_{E_p}= \Z_2+\frac{\sqrt{D}+1}{2}\Z_2$ and $\vol(\go_{E_p}^\times)=3\times \frac{2}{3}=2$ from the proof of Lemma~\ref{PeriodL1}. 
\end{proof}

Let $D>0$ and fix a complete set of representatives $\{\fa_j\}_{j=1}^h$ of the narrow ideal class group of $E$. Let $\a_j\in \A_{E,\fin}^\times$ be the idele corresponding to $\fa_j$. Let $\e_0>1$ be the fundamental unit of $E$. Set $\e=\e_0$ if $\nr(\e_0)=1$ and $\e=\e_0^2$ if $\nr(\e_0)=-1$. Then by the decomposition 
$$
\A_\Q^\times E^\times \bsl \A^\times_E/\hat \cO_E^\times=E^\times \bsl \A_E^1/\hat\cO_E^\times=\bigsqcup_{j=1}^{h}\a_j\cdot \e^{\Z}\bsl \tilde E_\infty^{1}
$$
with $\tilde E_{\infty}^1=\{\tau \in E_\infty^\times/\{\pm 1\} | \nr_{E/\Q}(\tau)=1\}$, using Lemmas~\ref{Period0712}, \ref{PeriodL1} and \ref{PeriodL3}, we have the following:
\begin{align*}
\Pcal_{\Delta}(\tilde\phi)&=2\times 
D^{1/2} \sum_{j=1}^{h} \int_{\e^{\Z}\bsl \tilde E_\infty^1} \tilde\phi(\Psi_{D}(\alpha_{j}) \Psi_{D}(\tau_\infty) \gamma_{D,\infty}) \,|f/2|_\infty\d^\times \dot \tau_\infty,
\end{align*} 
where $\d^\times \dot \tau_\infty$ is the measure on $\e^\Z\bsl \tilde E_\infty^1$ derived from the measure $\d^\times\dot \tau$ on $E^\times \A_\Q^\times\bsl \A_E^\times$. Let us fix a decomposition of the element $\Psi_{D}(\alpha_{j})$ along $\GL_2(\A_\Q)=\GL_2(\Q)\GL_2(\R)^{+}\GL_2(\hat \Z)$. This decomposition is explicitly obtained from an oriented $\Z$-basis of $\fa_j$ in the following way. Since $\Psi_D$ is optimal, there exists a $\Z$-basis $\{\omega_1,\omega_2\}$ of $\cO_{E}$ such that 
\begin{align}
[\tau \omega_1,\tau \omega_2]=[\omega_1,\omega_2]\Psi_D(\tau)
 \label{2017/12/17-3}
\end{align} 
for all $\tau \in E$. Let $\{\eta_1,\eta_2\}$ be a $\Z$-basis of $\fa_j$ (which is a free $\Z$-module of rank $2$). Since $\{\eta_1,\eta_2\}$ and $\{\omega_1,\omega_2\}$ are $\Q$-bases of $E$, we have a rational matrix $\gamma_j \in \GL_2(\Q)$ such that 
\begin{align}
[\eta_1,\eta_2]=[\omega_1,\omega_2]\gamma_j.
 \label{2017/12/17-5}
\end{align}
We suppose $\det (\gamma_{j,\infty})>0$. For any prime $p$, we have two $\Z_p$-bases $\{\eta_1,\eta_2\}$ and $\{\alpha_{j,p}\omega_1,\alpha_{j,p}\omega_2\}$ of the $\Z_p$-module $\fa_j \otimes \Z_p\cong \Z_p\oplus \Z_p$. Thus, there exists a matrix $k_{j,p}\in \GL_2(\Z_p)$ such that 
\begin{align}
[\alpha_{j,p}\omega_1,\alpha_{j,p}\omega_2]=[\eta_1,\eta_2] k_{j,p}\quad (\forall\, p).
 \label{2017/12/17-4}
\end{align}
From \eqref{2017/12/17-3}, \eqref{2017/12/17-4} and \eqref{2017/12/17-5}, 
\begin{align*}
[\omega_1,\omega_2]\Psi_{D}(\alpha_j)=[\omega_1,\omega_2]\gamma_j\,k_{j}
\end{align*}
with $k_j=(k_{j,p})_{p<\infty}\in \GL_2(\hat \Z)$. Thus, $\Psi_{D} (\alpha_j) =\gamma_j\,\gamma_{j,\infty}^{-1}\,k_j$ holds in the adelization $\GL_2(\A_\Q)$, where $\gamma_{j,\infty}$ is the image of $\gamma_j$ in $\GL_2(\R)^{+}$. 
Thus, 
\begin{align*}
\Pcal_{\Delta}(\tilde\phi) = 2 \times D^{1/2} \sum_{j=1}^{h} \int_{\e^{\Z}\bsl \tilde E_\infty^1} \tilde\phi(\gamma_{j,\infty}^{-1} \Psi_{D}(\tau_\infty) \gamma_{D,\infty})
\,|f/2|_\infty d^\times \dot \tau_\infty.
\end{align*}
Now set $\tau_{\infty}=x\oplus \delta y\in E_\infty^\times$ with $x,y \in \R$ (this is an expression in the quotient ring $\R[X]/(X^2-D)=\R\oplus \delta\R$ with $\delta$ being the class of $X$). Then the connected component $(E_\infty^\times)^0$ of $1$ in $E_\infty^\times$ is $\{\tau_\infty\in E_\infty|x^2-Dy^2>0, x-\sqrt{D}\,y>0\,\}$ and $E^{1}_\infty=\{\tau_\infty=x\oplus \delta y|\,x^2-Dy^2=1\}$. For $\tau_\infty=x\oplus \delta y\in E_\infty^\times$, set $t=x+\sqrt{D}y\in \R^\times$. Then $t$ (restricted to $E^1_\infty$) is a coordinate function on $E^{1}_\infty$, giving a group isomorphism $E^{1}_\infty\cong \R^\times$. Note $\tilde E_\infty^1=E_\infty^1/\{\pm 1\}\cong (E_\infty^1)^0$. We have the direct product decomposition $(E_\infty^\times)^0=\R_{+}^\times\,\tilde E_\infty^1$ as $\tau_{\infty}=r\tau_{\infty}^1$ with $r>0$ and $\tau_\infty^1\in (E_\infty^1)^0$. A computation reveals 
$$
\frac{\d r}{r} \wedge \frac{\d t}{t}=\sqrt{D}\frac{\d x\wedge \d y}{x^2-Dy^2},$$which shows that the measure $\d^1 \tau_{\infty}$ on $\tilde E^1_\infty$ derived from $|f/2|_\infty \d^\times \tau_\infty$ is $\tfrac{1}{\sqrt{D}}\frac{\d t}{|t|}$. 

Suppose $D\equiv 1\pmod{4}$. Then 
\begin{align*}
\Psi_{D}(\tau_\infty) \gamma_{D,\infty}=\gamma_{D,\infty} \left[\begin{smallmatrix} t & 0 \\ 0 & t^{-1} \end{smallmatrix} \right]
\end{align*}
for any $\tau_\infty=x\oplus \delta y \in \tilde E^1_\infty$ with $t=x+\sqrt{D}y$. Hence by the definition of $\tilde \phi$, 
$$
\tilde \phi(\gamma_{j,\infty}^{-1} \Psi_{D}(\tau_\infty)\gamma_{D,\infty})=\phi(\gamma_{j,\infty}^{-1}\gamma_{D,\infty}(it^2))
$$
and we obtain the expression
\begin{align*}
\Pcal_{\Delta}(\tilde \phi)= 2D^{1/2}\sum_{j=1}^{h} \int_{\e^{\Z}\bsl \tilde E_\infty^{1}}\phi(\gamma_{j,\infty}^{-1}\gamma_{D, \infty}(it^2))\,\tfrac{1}{\sqrt{D}}\tfrac{\d t}{|t|}.
\end{align*} Let $Q_D(x,y)$ be the element of $\Fcal(D)$ such that $Q_D(x,1)=0$ has the roots $\gamma_{D,\infty}(0)=\frac{-\sqrt{D}+1}{\sqrt{D}+1}$ and $\gamma_{D,\infty}(i\infty)=\frac{-\sqrt{D}-1}{\sqrt{D}-1}$, or explicitly 
$$
Q_D(x,y)=\tfrac{D-1}{4}x^2+\tfrac{D+1}{2}xy+\tfrac{D-1}{4}y^2. 
$$
Let $Q_j\in \Fcal(D)$ be an element that belongs to the same $\Q^\times$-homothety class as $Q_{D}\gamma_{j}$. Then $Q_j(x,1)=0$ has the two roots $\gamma_{j,\infty}^{-1}\gamma_{D,\infty}(0)$ and $\gamma_{j,\infty}^{-1}\gamma_{D,\infty}(i\infty)$ with $\det(\gamma_{j,\infty}^{-1}\gamma_{D,\infty})>0$. Hence
$$
\Pcal_{\Delta}(\tilde \phi)=\sum_{j=1}^{h} \int_{\Gamma(Q_j)\bsl \Omega_{Q_j}} \phi(z)\,|d_{Q_j}z|
$$ 
as desired. The case $D\equiv 0 \pmod {4}$ is similar.

Next, we consider the case $D<0$. Let $\{\fa_j\}_{j=1}^{h}$ be a complete set of representatives of the ideal class group of $E$ and $\a_j\in \A_{E,\fin}^{\times}$ the idele corresponding to $\fa_j$. Then 
\begin{align*}
\A_\Q^{\times} E^\times \bsl \A_{E}^\times/\widehat \cO_E^\times=\bigsqcup_{j=1}^{h}\a_j\cdot W_D\bsl E^1_\infty,
\end{align*}
where $W_D$ is the set of roots of unity in $E$. We have that $E_\infty^1=\{x_1+i\sqrt{|D|}y_1 \mid x_1^2+|D|y_1^2=1\}$ is isomorphic to the unit circle by $x_1=\cos\theta$, $y_1=\sqrt{|D|}^{-1}\sin \theta\,(0\leq \theta<2\pi)$. By the direct product decomposition $E_\infty^\times=\R_+^\times\,E_\infty^1$ with $\tau_\infty=r(\cos\theta+i\sin \theta)$, the measure $\d\tau_\infty$ is decomposed as
\begin{align*}
\zeta_{\C}(1)\frac{\d x\,\d y}{x^2+|D|y^2}=\frac{1}{\sqrt{|D|}}\frac{\d r}{r}\,\frac{\d \theta}{\pi}, 
\end{align*}
from which the measure $\d^1\tau_\infty$ induced on $E_\infty^1$ is seen to be equal to $\frac{1}{\sqrt{|D|}}\frac{\d \theta}{\pi}$. Suppose $D\equiv 1\pmod{4}$. A computation shows the equality 
\begin{align}
\Psi_{D}(\tau_\infty) \gamma_{D,\infty}=\gamma_{D,\infty} 
\left[\begin{smallmatrix}x & -\sqrt{|D|}y \\ \sqrt{|D|}y & x \end{smallmatrix}\right] \in \gamma_{D,\infty}\,Z(\R)^0\SO(2), \quad \tau_\infty=x+\sqrt{D}y\in E_\infty^\times.
 \label{PsiKinfty}
\end{align}
Let $\gamma_j\in \GL_2(\Q)$ and $\gamma_{j,\infty}\in \GL_2(\R)^+$ be as above. Since $\tilde\phi$ is right $Z(\R)\SO(2)$-invariant,  
\begin{align*}
\Pcal_{\Delta}(\tilde \phi)
&=2|D|^{1/2}\sum_{j=1}^{h} \int_{W_D\bsl E_\infty^1}
\tilde \phi(\gamma_{j,\infty}^{-1}\gamma_{D,\infty})\,\d^1\tau_\infty
\\
&=2|D|^{1/2} \sum_{j=1}^{h}\tilde \phi(\gamma_{j,\infty}^{-1}\gamma_{D,\infty})
\,\frac{1}{\# W_D} \int_{0}^{2\pi} \frac{1}{\sqrt{|D|}} \frac{\d \theta}{\pi}
=2\times\frac{2}{w_D} \sum_{j=1}^{h}\tilde \phi(\gamma_{j,\infty}^{-1}\gamma_{D,\infty}).
\end{align*}
For the element $Q_{D}(x,y)=\frac{D-1}{4}x^2+\frac{D+1}{2}xy+\frac{D-1}{4}y^2$ of $\Fcal(D)$, the quadratic equation $Q_D(z,1)=0$ has the unique root $\gamma_{D,\infty}(i)$ on $\fH$. Let $Q_{j}\in \Fcal(D)$ be an element belonging to the same $\Q^\times$-homothety class as $Q_{D}\gamma_{j}$. Then $Q_{j}(z,1)=0$ has the unique root $z_{Q_j}=\gamma_{j,\infty}^{-1}\gamma_{D,\infty}(i)$ on $\fH$. Hence, $\tilde \phi(\gamma_{j,\infty}^{-1}\gamma_{D,\infty})=\phi(z_{Q_j})$. This completes the proof of the equality $\Pcal_{\Delta}(\tilde \phi)=2\,\PP_{D}(\phi)$ for $D\equiv 1\pmod{4}$. The case $D\equiv 0 \pmod{4}$ is similar. 
\hfill $\square$

\medskip
\noindent
{\bf Remark} : The formula \eqref{ZagierEisen}
can be deduced from \cite[Proposition 7.7]{SugiyamaTsuzuki2018}
by Proposition \ref{PeriodL0}.

\subsection{Adelic trace formula}
Let $q$ be a prime number and $\Delta=f^2D$ a discriminant with $D\in \Dcal$, $f\in \N$. For $a\in \Q_q^\times$ and $s,z\in \C$, we set 
\begin{align}
 \Scal_{q}^{\Delta,(z)}(s;a)
&=-q^{-\frac{s+1}{2}}\frac{\zeta_{q}\left(s+\frac{z+1}{2}\right)\zeta_{q}\left(s+\tfrac{-z+1}{2}\right)}{L_{q}(s+1,\varepsilon_{D})}
|a|_q^{\frac{s+1}{2}}, \quad (|a|_q\leq 1), 
 \label{Scaldeltazsa-f1}
\\
\Scal_q^{\Delta,(z)}(s;a)
&=-q^{-\frac{s+1}{2}}
\left\{\frac{\zeta_{q}(-z)\zeta_{q}\left(s+\tfrac{z+1}{2}\right)}{L_{q}\left(\tfrac{-z+1}{2},\varepsilon_{D}\right)}\left|a\right|_q^{\frac{-z+1}{4}}
+\frac{\zeta_{q}(z)\zeta_{q}\left(s+\tfrac{-z+1}{2}\right)}{L_{q}\left(\tfrac{z+1}{2},\varepsilon_{D}\right)}\left|a\right|_q^{\frac{z+1}{4}}\right\}, \quad (|a|_q>1),
 \notag
\end{align}
where $\varepsilon_D$ is the idele class character of $\Q^\times$ corresponding to $\chi_D$.  

Fix a finite set $S$ of prime numbers and an even integer $k\geq 4$ once and for all, and let $\bfs=(s_q)_{q\in S}$ denote a variable belonging to the space $\fX_S=\prod_{q\in S} \fX_q$.
For $a \in \QQ^\times$, set
$$B^{\phi,S}(\Delta; a) =
\prod_{q \notin S}\biggl\{ \frac{\zeta_{q}(-\nu_p)}{L_{q}(\frac{-\nu_p+1}{2},\varepsilon_D)}|a|_q^{\frac{-\nu_p+1}{4}}
+\frac{\zeta_{q}(\nu_p)}{L_{q}(\frac{\nu_p+1}{2}, \varepsilon_D)}|a|_q^{\frac{\nu_p+1}{4}}\biggr\}.
$$
Recall that $\Qcal_\Q^{\rm{Irr}}$ is defined to be the set of all the equivalence classes $(t:n)_{\Q}$ of pairs $(t,n)\in \Q^2$ with $t^2-4n\in \Q^\times-(\Q^\times)^{2}$, where two such pairs $(t,n)$ and $(t',n')$ are defined to be equivalent if and only if $t'=ct$, $n'=c^2n$ $(\exists c\in \Q^\times)$, and that $\Qcal_{\Q}^{\rm Irr}\cap \Qcal_{\Q}^{S}$ is the set of all $(t:n)_{\Q}$ such that there exists a finite idele $c\in \A_\Q^\times$ with  $c_pt\in \Z_p$, $c_p^2n\in \Z_p^\times$ for all primes $p\not\in S$ (\cite[\S 7]{SugiyamaTsuzuki2018}). For $\e\in \{+,-\}$, let $Q_{\e}(S)$
be a complete set of representatives $(t, n)$ of the equivalence classes $(t:n)_\Q \in \Qcal_{\Q}^{{\rm Irr}}\cap \Qcal_\Q^{S}$ with $\e n>0$ and $t, n\in \ZZ$. Set 
\begin{align*}
\JJ_{k,S}^{\rm ell}(\bfs)=\tfrac{1}{2} \sum_{(t,n)\in Q_{+}(S)} \Pcal_{D}(\tilde \phi)\,{B}^{\phi,S}(\Delta; nf^{-2})\{\prod_{q\in S}\Scal_q^{\Delta,(\nu_q)}(s_q;nf^{-2})
\}\,\Ocal_k^{\sgn(\Delta),(\nu_\infty)}\left(\tfrac{t}{\sqrt{|\Delta|}}\right),
\end{align*}
where $\Delta=t^2-4n=f^2D$ for $(t,n)\in Q_{+}(S)$ with $f\in \N$, $D\in \Dcal$.
Define
\begin{align*}
\JJ_{k,S}^{\rm hyp}(\bfs)=\tfrac{1}{4}\hat L\left( \tfrac{1}{2},\phi \right)\sum_{a\in \Z(S)_+^\times-\{1\}} {B}^{\phi,S}(1; a(a-1)^{-2}) \{\prod_{q\in S} \Scal_q^{1,(\nu_q)}(s_q;a(a-1)^{-2})\}\,\Ocal_k^{+,(\nu_\infty)}\left(\tfrac{a+1}{a-1}\right),
\end{align*}
where $\Z(S)_+^\times$ denotes the set of rational numbers of the form $\prod_{q\in S}q^{m_q}$ with $m_q\in \Z$. 

For $\bfs=\{s_q\}_{q\in S}\in \fX_S$, set
$$
\II_{k,S}(\bfs)=\tfrac{4\pi}{k-1}\sum_{f \in H_k}{\mu_f(\phi)}
\frac{1}{\prod_{q\in S}(q^{1/2}\l_{f}(q)- (q^{(1+s_q)/2}+q^{(1-s_q)/2}))}.$$
\begin{prop}\label{ResolventTF}
The series $\JJ_{k,S}^{\rm hyp}(\bfs)$ and $\JJ_{k,S}^{\rm ell}(\bfs)$
 converge absolutely and locally uniformly for $\Re s_q\gg 0\,(q\in S)$. For all $\bfs\in \fX_S$ with $\Re s_q\gg 0\,(q\in S)$, we have the identity
\begin{align}
\II_{k,S}(\bfs)=\JJ_{k,S}^{\rm hyp}(\bfs)+\JJ_{k,S}^{\rm ell}(\bfs).
\label{RTF-s}
\end{align}
\end{prop}
\begin{proof}
We apply \cite[Theorem 10.1]{SugiyamaTsuzuki2018}\footnote{On the left-hand side of the formula in \cite[Theorem 10.1]{SugiyamaTsuzuki2018}, $(-1)^{\#S}$ should be removed.}
to the case $F=\Q$, $\fn=\Z$, $\xi=\pi_{\phi}$, noting that the theorem in this setting holds true for any $S$ by \S \ref{Appendix}. For any $f\in H_k$,
let $\pi_{f}$ be the representation of $\GL_2(\A_\Q)$ generated by $\tilde f$, the lift of $f$ to $\GL_2(\A_\Q)$ constructed as in \cite[\S 3.6]{Bump}. Then $\pi_{f}$ is irreducible and cuspidal and the set $\Pi_{\rm cus}(k,\Z)$ is exhausted by $\pi_{f}$ with $f\in H_k$. We notice that $\varphi_\xi^{0}$ in \cite[\S10]{SugiyamaTsuzuki2018} coincides with $2{\tilde \phi}$ (see the first paragraph of \S~\ref{Proof of Theorem MVT-L3-1}). Under the normalization of the Haar measure on $\GL_2(\A_\Q)$ used in \cite{SugiyamaTsuzuki2018}, we easily see $\langle \tilde \phi |\tilde f\bar{\tilde f}\rangle_{L^2}=\mu_{f}(\phi)$ and $\langle \tilde f|\tilde f_1\rangle_{L^2}=\langle f,f_1\rangle$ for $f,f_1\in H_k$. We note that if $(t,n)\in Q_{-}(S)$, then $|t/\sqrt{t^2-4n}|<1$, which implies $\Ocal_{k}^{+,(\nu_\infty)}(t/\sqrt{t^2-4n})=0$; thus the summand over $Q_{-}(S)$ is zero. Thus, the formula in \cite[Theorem 10.1]{SugiyamaTsuzuki2018} multiplied by $2^{-1}$ is precisely our \eqref{RTF-s}.  
\end{proof}

\subsection{The proof of Theorem~\ref{MVT-L3-1}}
We deduce Theorem~\ref{MVT-L3-1} from Proposition~\ref{ResolventTF}. Recall the Chebyshev polynomials of the 2nd kind $U_m(X)\,(m\in \N_0)$ are defined by the relation $U_m(\cos \theta)=\sin((m+1)\theta)/\sin \theta$. Let $\d\mu_q(s)$ denote the holomorphic $1$-form $2^{-1}(\log q)(q^{(1+s)/2}-q^{(1-s)/2})\,\d s$ on $\fX_q$, and let $L_q(c)$ be the contour $y\mapsto c+iy\,(y\in \R,\,|y|\leq \frac{2\pi i }{\log q})$ on $\fX_q$ for a sufficiently large $c \in \RR$. 

\begin{lem}\label{MVT-L2}
	Suppose $c>(|\Re z|-1)/2$.
	When $a \in \QQ_q^\times$ satisfies $|a|_q\le 1$, we have the equality
	\begin{align*}
	&\tfrac{-1}{2\pi i}\int_{L(c)}(-q^{-\frac{s+1}{2}})\frac{\zeta_{q}(s+\frac{z+1}{2})\zeta_{q}(s+\frac{-z+1}{2})}
	{L_{q}(s+1,\varepsilon_D)}|a|_q^{\frac{s+1}{2}}U_m(2^{-1}(q^{-s/2}+q^{s/2}))d\mu_q(s) \\
	= & \delta(m-\ord_q(a) \in 2\NN_0)q^{-m/2}\biggl\{ \frac{\zeta_{q}(-z)q^{(m-\ord_q(a))(-z+1)/4}}{L_{q}(\frac{-z+1}{2},\varepsilon_D)} +\frac{\zeta_{q}(z)q^{(m-\ord_q(a))(z+1)/4}}{L_{q}(\frac{z+1}{2}, \varepsilon_D)}\biggr\}.
	\end{align*}
		When $a \in \QQ_q^\times$ satisfies $|a|_q>1$, we have
	\begin{align*}
	&\tfrac{-1}{2\pi i}\int_{L(c)}(-q^{-\frac{s+1}{2}})\left\{ \frac{\zeta_{q}(-z)\zeta_{q}(s+\frac{z+1}{2})}
	{L_{q}(\frac{-z+1}{2},\varepsilon_D)}|a|_q^{(-z+1)/4}+ \frac{\zeta_{q}(z)\zeta_{q}(s+\frac{-z+1}{2})}
	{L_{q}(\frac{z+1}{2},\varepsilon_D)} |a|_q^{(z+1)/4} \right\} \\
	& \times U_m(2^{-1}(q^{-s/2}+q^{s/2}))d\mu_q(s) \\
	= &
	\delta(m \in 2\NN_0)q^{-m/2}\biggl\{ \frac{\zeta_{q}(-z)q^{(m-\ord_q(a))(-z+1)/4}}{L_{q}(\frac{-z+1}{2},\varepsilon_D)} +\frac{\zeta_{q}(z)q^{(m-\ord_q(a))(z+1)/4}}{L_{q}(\frac{z+1}{2}, \varepsilon_D)}\biggr\}.
	\end{align*}
\end{lem}
\begin{proof}
	Suppose $|a|_q\le 1$. By the change of variables $\xi=q^{-s/2}$, the integral is transformed to
	\begin{align*}
	& \tfrac{-1}{2\pi i}\oint_{|\xi|=q^{-c/2}}
	\frac{1-\xi^2 \chi_D(q)q^{-1}}{(1-\xi^2 q^{(-z-1)/2})(1-\xi^2q^{(z-1)/2})}|a|_q^{1/2}\xi^{\ord_q(a)}(\xi^{m+1}-\frac{1}{\xi^{m+1}}) d\xi\\
	=& \Res_{\xi=0}\frac{1-\xi^2 \chi_D(q)q^{-1}}{(1-\xi^2 q^{(-z-1)/2})(1-\xi^2q^{(z-1)/2})}|a|_q^{1/2}\xi^{\ord_q(a)-m-1}.
	\end{align*}
	When $\ord_q(a)\ge m+1$, this quantity vanishes. When $0\le \ord_q(a)\le m$, the residue as above vanishes if $m-\ord_q(a)$ is odd.
	If $m-\ord_q(a)=2b$ for $b \in \NN_0$, then
	we have
{\allowdisplaybreaks\begin{align*}
	& \Res_{\xi=0}\frac{1-\xi^2 \chi_D(q)q^{-1}}{(1-\xi^2 q^{(-z-1)/2})(1-\xi^2q^{(z-1)/2})}|a|_q^{1/2}\xi^{\ord_q(a)-m-1} \\
	= & |a|_q^{1/2}\Res_{\xi=0}\biggl\{ \sum_{l=0}^\infty \sum_{j=0}^\infty \xi^{2l+2j} q^{\frac{-z-1}{2}l+\frac{z-1}{2}j}
	-\chi_D(q)q^{-1}
	\sum_{l=0}^{\infty}\sum_{j=0}^{\infty} \xi^{2l+2j+2} q^{\frac{-z-1}{2}l+\frac{z-1}{2}j}
	\biggr\} \xi^{-2b-1} \\
	= & |a|_q^{1/2}\biggl\{\sum_{l=0}^{b} q^{\frac{-z-1}{2}l +\frac{z-1}{2}(b-l)}
	-\chi_D(q)q^{-1}
	\sum_{l=0}^{b-1}q^{\frac{-z-1}{2}l+\frac{z-1}{2}(b-1-l)}
	\biggr\} \\
	= & |a|_q^{1/2}\frac{1-q^{-(b+1)z}}{1-q^{-z}}q^{\frac{z-1}{2}b}-\chi_D(q)q^{-1}\delta(b \ge 1)\frac{1-q^{-bz}}{1-q^{-z}}q^{\frac{z-1}{2}(b-1)} \\
	= & \delta(b=0)|a|_q^{1/2} +\delta(b \ge 1)|a|_q^{1/2}\biggl\{ \frac{1-\chi_{D}(q)q^{(-z-1)/2}}{1-q^{-z}}q^{b(z-1)/2} +\frac{1-\chi_{D}(q)q^{(z-1)/2}}{1-q^z} q^{b(-z-1)/2}\biggr\} \\
	=& q^{-m/2}\biggl\{\frac{\zeta_{q}(z)}{L_{q}(\frac{z+1}{2},\varepsilon_D)}q^{\frac{z+1}{4}(m-\ord_q(a))} +\frac{\zeta_{q}(-z)}{L_{q}(\frac{-z+1}{2}, \varepsilon_D)}q^{\frac{-z+1}{4}(m-\ord_q(a))}\biggr\}.
	\end{align*}}This yields the first identity.
	The case where $|a|_q>1$ is more easily handled by noting $\Res_{\xi=0}(\frac{\xi^{-m-1}}{1-\xi^2 q^{(-z-1)/2}}) = \delta(m \in 2\NN_0)q^{m(-z-1)/4}$.
\end{proof}

For a positive integer $n=\prod_{q\in S}q^{m_q}$ with $m_q\in \N$, let $\a_n$ be a function on $\fX_S$ defined as $\a_n(\bfs)=\prod_{q\in S}U_{m_q}(2^{-1}(q^{-s_q/2}+q^{s_q/2}))$.
By Proposition~\ref{ResolventTF}, we obtain an equality between 
\begin{align}
\left(\tfrac{-1}{2\pi i}\right)^{\# S} \int_{L_S(c)}\II_{k,S}(\bfs)\a_n(\bfs)\d\mu_S(\bfs)
 \label{trace-LHS}
\end{align}
and 
\begin{align}
\left(\tfrac{-1}{2\pi i}\right)^{\# S} \int_{L_S(c)}\JJ_{k,S}^{\rm hyp}(\bfs)\a_n(\bfs)\d\mu_S(\bfs)+
\left(\tfrac{-1}{2\pi i}\right)^{\# S} \int_{L_S(c)}\JJ_{k,S}^{\rm ell}(\bfs)\a_n(\bfs)\d\mu_S(\bfs), 
 \label{trace-RHS}
\end{align}
where $\int_{L_S(c)}f(\bfs)\d\mu_S(\bfs)$ is defined as the iteration of the one-dimensional contour integral $\int_{L_{q(j)}(c)}\,\d\mu_{q(j)}(s_{q(j)})\,(j=1,\dots,l)$ in any ordering $S=\{q(1),\dots,q(l)\}$. 
By the formula 
$$
\tfrac{-1}{2\pi i }\int_{L_q(c)}\{q^{(1+\nu)/2}+q^{(1-\nu)/2}-q^{(1+s)/2}-q^{(1-s)/2}\}^{-1}\a(s)\d\mu_q(s)=\a(\nu), 
$$
which is valid for any holomorphic function $\a$ on $\fX_q$ such that $\a(-s)=\a(s)$, the quantity \eqref{trace-LHS} becomes 
\begin{align*}
\tfrac{4\pi}{k-1}\sum_{f \in H_k}\{\prod_{q\in S}U_{m_q}(2^{-1}(q^{\nu_q(f)/2}+q^{-\nu_q(f)/2}))\}\,{\mu_f(\phi)},
\end{align*}
where $\nu_q(f)\in \fX_q$ is determined by $\l_{f}(q)=q^{\nu_q(f)/2}+q^{-\nu_q(f)/2}$. Since $$\l_{f}(n)=\prod_{q\in S} U_{m_q}(2^{-1}(q^{\nu_q(f)/2}+q^{-\nu_q(f)/2})),$$ the above quantity multiplied by $n^{1/2}$ is equal to the left-hand side of \eqref{MVT-L3-1-f1}. 

The first term of \eqref{trace-RHS} becomes
\begin{align}
&\tfrac{1}{4}\hat L\left(\tfrac{1}{2},\phi\right)
\sum_{a\in \Z(S)_+^\times-\{1\}}
{\bf B}^{S}(\nu; (a-1)^2)
\Ocal_{k}^{+,(\nu_\infty)}\left(\tfrac{a+1}{a-1}\right)
 \notag
\\
&\quad \times\prod_{q\in S} \tfrac{-1}{2\pi i } \int_{L_q(c)}\Scal_q^{1,(\nu_q)}\left(s;
\tfrac{a}{(a-1)^2}\right) U_{m_q}(2^{-1}(q^{s/2}+q^{-s/2}))\,\d\mu_q(s).
\label{MVT-L4-f1} 
\end{align}
By Lemma \ref{MVT-L2}, the $S$-factor \eqref{MVT-L4-f1} is non-zero only if\begin{align}
m_q \geq \ord_{q}((a-1)^{-2}a)\geq 0, \quad m_q \equiv \ord_{q}(a)\pmod{2}
 \label{MVT-L4-f2}
\end{align}
or 
\begin{align}
\ord_{q}((a-1)^{-2}a)< 0, \quad m_q \equiv 0 \pmod{2}
 \label{MVT-L4-f3}
\end{align}
for all $q\in S$. Let $a\in \Z(S)_+^\times-\{1\}$ be an element satisfying these two conditions, and write $a=a_1a_2^{-1}$ with relatively prime integers $a_1,a_2>0$.

Then there exists  $b \in \N$ such that $n=a_1 a_2 b^2$. To prove this, it suffices to check $\ord_{q}(n/(a_1a_2)) \in 2 \N_{0}$
for all prime numbers $q$. If $q \notin S$, then $\ord_q(n)=\ord_q(a_1a_2)=0$ holds.
If $q\in S$ and $q|a_1a_2$, then we have $\ord_{q}(a_1a_2)\geq 1$, and $\ord_{q}(a_1-a_2)=0$ due to $(a_1,a_2)=1$. Hence \eqref{MVT-L4-f3} can not happen
and \eqref{MVT-L4-f2} yields $\ord_{q}(n/(a_1a_2)) \in 2 \N_{0}$. If
$q \in S$ and $(q,a_1a_2)=1$, then \eqref{MVT-L4-f2} or \eqref{MVT-L4-f3} implies that $\ord_q(n/(a_1a_2))=m_q \in 2\N$.


If we set $d_1=a_1b$ and $d_2=a_2b$, then $n=d_1d_2$ and $a=d_1d_2^{-1}$ holds. By Lemma~\ref{MVT-L2}, the $q$-factor of \eqref{MVT-L4-f1} becomes
\begin{align*}
q^{-m_q/2}\left\{ \frac{\zeta_q(-\nu_q)}{L_q\left(\tfrac{-\nu_q+1}{2},\varepsilon_{D}\right)}|d_1-d_2|_q^{(\nu_q-1)/2}+
\frac{\zeta_q(\nu_q)}{L_q\left(\tfrac{\nu_q+1}{2},\varepsilon_{D}\right)}
|d_1-d_2|_q^{(-\nu_q-1)/2}\right\},
\end{align*}
whose product over $q\in S$ together with ${\bf B}^{S}(\nu;(a-1)^2)={\bf B}^{S}(\nu;(d_1-d_2)^2)$ yields ${\bf B}(\nu;(d_1-d_2)^{2})$. Thus we obtain the first term of the right-hand side of \eqref{MVT-L3-1-f1} up to the factor $n^{1/2}$. 

Let us examine the second term of \eqref{trace-RHS}. We need a lemma.

\begin{lem} Let $n=\prod_{q\in S}q^{m_q}$ be a positive integer with $m_q\in \N$. Let $(t:n')\in Q_{+}(S)$ be such that 
	\begin{align}
		\prod_{q\in S}
		\tfrac{-1}{2\pi i } \int_{L_q(c)}\Scal_q^{t^2-4n',(\nu_q)}\left(s;
		\tfrac{n'}{f^2}\right) U_{m_q}(2^{-1}(q^{s/2}+q^{-s/2}))\,\d\mu_q(s)\not=0,
		\label{Scalnonvanishing}
	\end{align}
	where $t^2-4n'=f^2D$ with $f\in \N$ and a fundamental discriminant $D$. Then there exists $b\in \Z(S)^\times$ such that $bt\in \Z_q$ and $\ord_{q}(b^{2}n')=m_q$ for all $q\in S$.
\end{lem}
\begin{proof} For each $q\in S$, we shall show that $\ord_{q}(n')\equiv m_q\pmod{2}$ and the element $b_q=q^{(m_q-\ord_q(n'))/2}$ satisfies both $b_qt \in \Z_q$ and $b_q^2n'\in n\Z_q^{\times}$. Then $b=\prod_{q\in S}b_q$ meets the requirements. In what follows, $q$ denotes an element of $S$.
	
	(1) Suppose $q\neq 2$, $\ord_q(t^2-4n')\in 2\Z$ and $t^2-4n'\not\in (\Q_q^\times)^2$. Then we have $|t|_q^2\leq |4n'|_q$ and
	$|\tfrac{n'}{f^2}|_q\geq 1$ from \cite[Proposition 7.12 (1)]{SugiyamaTsuzuki2018}.
	By Lemma~\ref{MVT-L2}, we have $m_q\in 2\NN$.
	If $|t|_q^2<|4n'|_q$, then $\ord_q(n') = \ord_q((t^2-4n')/4) \in 2\ZZ$.
If $|t|_q^2=|4n'|_q$, then $\ord_q(n')=\ord_q(t^2/4)\in 2\ZZ$ holds.	
Hence we obtain $\ord_q(n')\equiv 0 \equiv m_q \pmod{2}$, and $|b_q t|_q^2\leq |4b_q^2 n'|_q=|4n|_q\leq 1$.

(2) Suppose $q=2$, $\ord_2(t^2-4n')\in 2\Z$ and $\QQ_2(\sqrt{D}) = \QQ_2(\sqrt{5})$. Then,
$|t^2|_2\le |n'|_2$ and
$$|\tfrac{n'}{f^2}|_2 = \begin{cases}|\frac{1}{4}|_2=4>1  & (|t^2|_2<|4n'|_2 \text{ or } \frac{t^2}{n'}\in \ZZ_2^\times), \\
|\frac{1}{4}\frac{4n'}{t^2-4n'}|_2 \ge 4>1 & (\frac{t^2}{4n'} \in \ZZ_2^\times)
\end{cases}$$
by \cite[Proposition 7.12 (1)]{SugiyamaTsuzuki2018}. We note that $|2n'|_2 = |t^2|_2$
never happens by \cite[Lemma 7.11 (4)]{SugiyamaTsuzuki2018}.
Combining this with $\ord_2(t^2-4n')\in 2\ZZ$, $\ord_2(n')$ is even.
Thus we obtain $\ord_2(n')\equiv 0 \equiv m_2 \pmod{2}$ by Lemma~\ref{MVT-L2},
and $|b_2 t|_2^2 \le |b_2^2 n'|_2 = |n|_2\le1$.

(3) Suppose $q=2$, $\ord_2(t^2-4n')\in 2\Z$ and $\QQ_2(\sqrt{D})=\QQ_2(\sqrt{\tau})$ with $\tau=-1,-5$. Then,
$|t^2|_2 \le |2n'|_2$ and
$$|\tfrac{n'}{f^2}|_2 = \begin{cases}1  & (|t^2|_2<|4n'|_2), \\
	\frac{1}{2} & (\frac{t^2}{2n'} \in \ZZ_2^\times), \\
	|\frac{4n'}{t^2-4n'}|_2 \ge 1 & (\frac{t^2}{4n'} \in \ZZ_2^\times)
\end{cases}$$
by \cite[Proposition 7.12 (1)]{SugiyamaTsuzuki2018}.
As for the first two cases, $m_2\equiv \ord_{2}(n') \pmod{2}$ follows from Lemma \ref{MVT-L2}.
As for the third case, with the aid of Lemma \ref{MVT-L2},
$\ord_2(m_2) \equiv0 \equiv \ord_2(t^2/4) = \ord_2(n')$.
Hence we obtain $m_2\equiv \ord_{2}(n') \pmod{2}$,
and $|b_2t|_2^2 \le |2b_2^2n'|_2 = |2n|_2 <1$.
	
	(4) Suppose $\ord_{q}(t^2-4n')\in 1+2\Z$. Then $|t|_q^2\leq |4n'|_q$ from \cite[Proposition 7.12 (2)]{SugiyamaTsuzuki2018}. First suppose $|t|_q^2<|4n'|_q$. Then $|\tfrac{n'}{f^2}|_q=q^{-1}<1$ holds from \cite[Proposition 7.12 (2)]{SugiyamaTsuzuki2018}. By Lemma~\ref{MVT-L2}, we have $m_q\equiv \ord_q(\tfrac{n'}{f^2}) \equiv \ord_q(n')\pmod{2}$, and $|b_qt|_q^2<|4b_q^2 n'|_q=|4n|_q\leq 1$. Next suppose $|t|_q^2=|4n'|_q$. Then $|\tfrac{n'}{f^2}|_q\geq 1$ holds from \cite[Proposition 7.12 (2)]{SugiyamaTsuzuki2018} and
	$\ord_{q}(t^2-4n')\in 1+2\Z$. By Lemma~\ref{MVT-L2}, $m_q\equiv 0 \equiv \ord_{q}(t^2/4)= \ord_q(n') \pmod{2}$. By $|b_q t|_q^2=|4b_q^2 n'|_q=|4n|_q\le 1$, we are done.
%
\end{proof}

From this lemma, the second term of \eqref{trace-RHS} becomes
\begin{align}
&\tfrac{1}{2}
\sum_{D \in \Dcal}\sum_{t \in\Tcal(n, D)}
\Pcal_{D}(\phi)B^{\phi, S}(\nu; nf^{-2})
\Ocal_{k}^{\sgn(t^2-4n),(\nu_\infty)}\left(\tfrac{t}{\sqrt{|t^2-4n|}}\right)
 \notag
\\
&\quad \times\prod_{q\in S} \tfrac{-1}{2\pi i } 
\int_{L_q(c)}\Scal_q^{t^2-4n,(\nu_q)}\left(s;
nf^{-2}\right) U_{m_q}(2^{-1}(q^{s/2}+q^{-s/2}))\,\d\mu_q(s),
\end{align}
where $\Tcal(n,D)$ is the set of $t\in \Z$ such that $t^2-4n=f^2D$ with some $f\in \N$.
From Lemma~\ref{MVT-L2}, the $S$-factor is computed to be 
$$
n^{-1/2}\prod_{q\in S} \left\{ \frac{\zeta_q(-\nu_q)}{L_q\left(\tfrac{-\nu_q+1}{2},\varepsilon_{D}\right)}|f^{-2}|_q^{(-\nu_q+1)/4}+
\frac{\zeta_q(\nu_q)}{L_q\left(\tfrac{\nu_q+1}{2},\varepsilon_{D}\right)}
|f^{-2}|_q^{(\nu_q+1)/4}\right\},  
$$
which together with $B^{\phi,S}(\nu;nf^{-2})$ yields ${\bf B}(\nu; t^2-4n)$. Thus we obtain the second term of the right-hand side of \eqref{MVT-L3-1-f1}
up to the factor $n^{1/2}$ by using Proposition~\ref{PeriodL0}. This completes the proof of Theorem~\ref{MVT-L3-1}. \qed

\section{Proofs of Theorems~\ref{ASYMPT} and \ref{ASYMPT_average}}
Let $\phi$ be as before with the spectral parameter $(\nu_\infty, \nu=\{\nu_p\}_{p<\infty})$. 

Let us recall the function used in \cite{Zagier}. For $\Delta,t\in \R$ such that $t^2>\Delta\neq0$ and $s\in \C$ with $1-k<\Re s<k$, it is defined as the integral  
$$
I_{k}(\Delta,t;s)=\frac{\Gamma\left(k-\tfrac{1}{2}\right)\sqrt{\pi}}{\Gamma(k)}\int_{0}^{\infty}\frac{y^{k+s-2}}{(y^2+ity-\Delta/4)^{k-1/2}}\,\d y.
$$
Then from formulas on \cite[p.134]{Zagier}, if $\Delta=t^2-4n<0$, 
\begin{align}
\Ocal_k^{-,(z)}\left(\tfrac{t}{\sqrt{|\Delta|}}\right)
=|\Delta|^{-k/2}\, 2^{2k-1}\,(-1)^{k/2}\,n^{k/2}\,\{I_{k}\Bigl(-4,\tfrac{2t}{\sqrt{|\Delta|}};\tfrac{z+1}{2}\Bigr)+I_{k}\Bigl(-4,\tfrac{-2t}{\sqrt{|\Delta|}};\tfrac{z+1}{2}\Bigr)\}, 
\label{Zagierftn1}
\end{align}
and if $\Delta=t^2-4n>0$ and $n>0$, 
\begin{align}
\Ocal_{k}^{+,(z)}\left(\tfrac{t}{\sqrt{\Delta}}\right)=\frac{\Gamma\left(\tfrac{z+3}{4}\right)}{\sqrt{\pi}\Gamma\left(\tfrac{z+1}{4}\right)}\,
|\Delta|^{-k/2}\, 2^{2k-1}\,(-1)^{k/2}\,n^{k/2}\,\{I_{k}\left(4,\tfrac{2t}{\sqrt{\Delta}};\tfrac{z+1}{2}\right)+I_{k}\left(4,\tfrac{-2t}{\sqrt{\Delta}};\tfrac{z+1}{2}\right)\}. 
\label{Zagierftn2}
\end{align}

\begin{lem}\label{OIestL1} If $k>(|\Re z|+1)/2$, then
\begin{align*}
\Ocal_{k}^{-,(z)}(0)=2^{k-1}(-1)^{k/2}\sqrt{\pi}\Gamma\left(\tfrac{k}{2}+\tfrac{-z-1}{4}\right)\Gamma\left(\tfrac{k}{2}+\tfrac{z-1}{4}\right)\Gamma(k)^{-1}. 
\end{align*}
\end{lem}
\begin{proof}
From \eqref{Zagierftn1},  we have  
\begin{align*}
\Ocal_k^{-,(z)}(0)&=4^{-k/2}\, 2^{2k-1}\,(-1)^{k/2}\frac{\Gamma\left(k-1/2\right)\sqrt{\pi}}{\Gamma(k)} 
2\int_{0}^{\infty}\frac{y^{k+({z-3})/{2}}}{(y^2+1)^{k-1/2}}\d y.
\end{align*}
By the change of variables $x=(1+y^2)^{-1}$, the $y$-integral becomes a half of the beta integral $B(\frac{k}{2}+\frac{-z-1}{4},\frac{k}{2}+\frac{z-1}{4})=\Gamma(\frac{k}{2}+\frac{-z-1}{4})\Gamma(\frac{k}{2}+\frac{z-1}{4}){\Gamma(k-1/2)^{-1}}$.\end{proof}

\begin{lem} \label{2017/12/04-2}
	Let $t,n\in \Z$ such that $\Delta=t^2-4n<0$. Set $\rho=2^{-1}(-t+i\sqrt{|\Delta|})$. For all $z\in i\R$, we have
	\begin{align*}
	\Ocal_{k}^{-,(z)}\left(\tfrac{t}{\sqrt{|\Delta|}}\right)
	=&\frac{2\pi}{i\sqrt{|\Delta|}}\frac{\Gamma(k+(z-1)/2)\Gamma(k+(-z-1)/2)}{\Gamma(k)^2}\\
	&\times \left\{
	\rho\,(\bar \rho/\rho)^{k/2}
	{}_2F_1\left(\tfrac{1+z}{2},\tfrac{1-z}{2};k;\tfrac{\bar \rho}{-i\sqrt{|\Delta|}}\right)
	-
	\bar \rho\,(\rho/\bar \rho)^{k/2}
	{}_2F_1\left(\tfrac{1+z}{2},\tfrac{1-z}{2};k;\tfrac{\rho}{i\sqrt{|\Delta|}}\right) \right\}.
	\end{align*}
\end{lem}
\begin{proof}
	By the definition of the associated Legendre function of the 1st kind (cf.\ \cite[8.702]{Gradshteyn}),
	$${\frak P}_{\tfrac{z-1}{2}}^{1-k}(ia)=(\tfrac{ia+1}{ia-1})^{(1-k)/2}\Gamma(k)^{-1} {}_2F_1(\tfrac{1-z}{2},\tfrac{z+1}{2}; k; \tfrac{1-ia}{2}).$$
	For $a=\frac{t}{\sqrt{|\Delta|}}$, we have $1-ia=\frac{2{\bar \rho}}{-i\sqrt{|\Delta|}}$ and
	\begin{align*}
	\sgn(a)i\sqrt{1+a^2}{\frak P}_{\tfrac{z-1}{2}}^{1-k}(ia)
	= & (ia+1)^{1/2}(ia-1)^{1/2} (\tfrac{ia+1}{ia-1})^{(1-k)/2}\Gamma(k)^{-1} {}_2F_1(\tfrac{1-z}{2},\tfrac{z+1}{2}; k; \tfrac{1-ia}{2})\\
	= & 2\rho (i\sqrt{|\Delta|})^{-1}(\bar\rho/\rho)^{k/2} \Gamma(k)^{-1} {}_2F_1(\tfrac{1-z}{2},\tfrac{1+z}{2}; k; \tfrac{\bar \rho}{-i\sqrt{|\Delta|}}).
	\end{align*}
	Here $\arg(ia\pm1)$ is taken so that $\arg(ia\pm1) \in(-\pi, \pi)$, which leads us to $(ia+1)^{1/2}(ia-1)^{1/2}=\sgn(a)i\sqrt{a^2+1}$.
	Similarly, we have
	\begin{align*}
	\sgn(a)i\sqrt{1+a^2}{\frak P}_{\tfrac{z-1}{2}}^{1-k}(-ia)
	= & 2\bar \rho(i\sqrt{|\Delta|})^{-1}(\rho/\bar\rho)^{k/2} \Gamma(k)^{-1} {}_2F_1(\tfrac{1-z}{2},\tfrac{1+z}{2}; k; \tfrac{\rho}{i\sqrt{|\Delta|}}).
	\end{align*}
	By substituting these, we have the desired formula of $\Ocal_k^{-, (z)}(a)$.
\end{proof}

\begin{prop} \label{2017/12/04-3}
	Let $t,n\in \Z$ be such that $\Delta=t^2-4n<0$. Set $\rho=2^{-1}(-t+i\sqrt{|\Delta|})$. Let $z\in i\R$. Then we have the asymptotic formula
	\begin{align*}
	\Ocal_{k}^{-,(z)}\left(\tfrac{t}{\sqrt{|\Delta|}}\right)/\Ocal_{k}^{-,(z)}(0)
	&= (i\sqrt{|\Delta|})^{-1} \,(-1)^{k/2}
	\left\{\rho\,(\bar \rho/\rho)^{k/2}
	-\bar \rho\,(\rho/\bar \rho)^{k/2}\right\}
	+O(k^{-1}), \qquad k\rightarrow \infty.
	\end{align*}
\end{prop}
\begin{proof}
	This follows from Lemmas~ \ref{OIestL1} and \ref{2017/12/04-2} by applying the asymptotic formulas
	\begin{align}\label{asym_gamma}
	&\frac{\Gamma(k+\a)}{\Gamma(k+\b)}=k^{\a-\b}(1+O(k^{-1})), \qquad{}_2F_1(a,b;k,x)=1+O(k^{-1})
	\end{align}
	found on \cite[p.12]{MOS} and \cite[Appendix A, Theorem 1]{Hejhal}.
\end{proof}

\begin{lem}\label{OIestL3} For $\e>0$ and $z\in i\R$, we have
\begin{align*}
|I_k\left(4,a;\tfrac{z+1}{2}\right)|
\ll_{z, \e} |\Gamma\left(\tfrac{k}{2}+\tfrac{-z-1}{4}\right)\Gamma\left(\tfrac{k}{2}+\tfrac{z-1}{4}\right)\Gamma(k)^{-1}|\,2^{k}\sqrt{k}a^{-(k-\e-1/2)}
\end{align*}
uniformly in $k$ and $|a|>2$. 
\end{lem}
\begin{proof}
We have
\begin{align*}
|I_k\left(4,a;\tfrac{z+1}{2}\right)|
&\leq \sqrt{\pi}\frac{\Gamma(k-1/2)}{\Gamma(k)} 
\int_{0}^{\infty} \frac{y^{k-1/2}}{|y^2-1+iay|^{k-1/2}} \d^\times y. 
\end{align*}
By using $y^2+y^{-2}\geq 2$ and $|a|>2$, the integral is estimated as
{\allowdisplaybreaks \begin{align*} 
y^{-(k-1/2)}|y^2-1+iay|^{k-1/2}&=
(y^{2}+y^{-2}+a^2-2)^{(k-1/2)/2} \\
&=(y^2+y^{-2}+a^2-2)^{(k-\e-1/2)/2} \times (y^2+y^{-2}+a^2-2)^{\e}
\\
&\geq |a|^{k-\e-1/2}\,(y^2+y^{-2}+2)^{\e/2}.
\end{align*}
}Thus
\begin{align*}
|I_k\left(4,a;\tfrac{z+1}{2}\right)|
&\leq \sqrt{\pi}\frac{\Gamma(k-1/2)}{\Gamma(k)} |a|^{-(k-\e-1/2)}
\int_{0}^{\infty} \frac{\d^\times y}{(y+y^{-1})^{\e/2}}. 
\end{align*}
By the duplication formula $\Gamma(k-1/2)=2^{k-3/2}\sqrt{\pi}^{-1}\Gamma(k/2-1/4)\Gamma(k/2+1/4)$ and \eqref{asym_gamma},
it is seen that 
\begin{align*}
\frac{\Gamma(k-1/2)}{\Gamma\left(\tfrac{k}{2}+\tfrac{-z-1}{4}\right)\Gamma\left(\tfrac{k}{2}+\tfrac{z-1}{4}\right)}=2^{k-3/2}\sqrt{\pi}^{-1}\times \frac{\Gamma(k/2-1/4)}{\Gamma\left(\tfrac{k}{2}+\tfrac{-z-1}{4}\right)}\times \frac{\Gamma(k/2+1/4)}{\Gamma\left(\tfrac{k}{2}+\tfrac{z-1}{4}\right)}
\end{align*}
with growing $k$ is majorized by $2^{k-3/2}\times |(\tfrac{k}{2})^{z/4}| \times |(\tfrac{k}{2})^{-z/4+1/2}|=2^{k-2}k^{1/2}$.
\end{proof}

\begin{lem}\label{OIestL4} Let $d_1>d_2$ be two positive integers. Then, for $z \in i\R$, 
\begin{align*}
I_k\left(4,2\tfrac{d_1+d_2}{d_1-d_2};\tfrac{z+1}{2}\right)
&=2^{2-2k}\pi \frac{\Gamma\left(k+\tfrac{z-1}{2}\right)\Gamma\left(k+\tfrac{-z-1}{2}\right)}{\Gamma(k)^2}i^{1-2k}
\\
&\quad \times \left(\tfrac{1}{i}\tfrac{d_1^{1/2}+d_2^{1/2}}{d_1^{1/2}-d_2^{1/2}}\right)^{-k-\frac{z-1}{2}}
{}_2F_1\left(k-\tfrac{1}{2},k+\tfrac{z-1}{2};2k-1;\tfrac{4(d_1d_2)^{1/2}}{(d_1^{1/2}+d_2^{1/2})^2}\right).
\end{align*}
\end{lem}
\begin{proof} Set $r_j=\sqrt{d_j}\,(j=1,2)$. 
Since $y^{2}+2i\frac{d_1+d_2}{d_1-d_2}y-1=(y+i\frac{r_1-r_2}{r_1+r_2})(y+i\frac{r_1+r_2}{r_1-r_2})$, by using the formula \cite[3.259.3]{Gradshteyn}, we have 
{\allowdisplaybreaks\begin{align*}
I_{k}\left(4,2\tfrac{d_1+d_2}{d_1-d_2};\tfrac{z+1}{2}\right)
=&\frac{\sqrt{\pi}\Gamma\left(k-\tfrac{1}{2}\right)}{\Gamma(k)} 
\int_{0}^{\infty} y^{k+\frac{z-3}{2}}\left(y+i\tfrac{r_1-r_2}{r_1+r_2}\right)^{1/2-k}\left(y+i\tfrac{r_1+r_2}{r_1-r_2}\right)^{1/2-k}\,\d y
\\
=&\frac{\sqrt{\pi}\Gamma\left(k-\tfrac{1}{2}\right)}{\Gamma(k)}
B\left(k+\tfrac{z-1}{2},k+\tfrac{-z-1}{2}\right) i^{1-2k}
\\
&\times \left(\tfrac{1}{i}\tfrac{r_1+r_2}{r_1-r_2}\right)^{-k-(z-1)/2}{}_2F_1\left(k-\tfrac{1}{2},k+\tfrac{z-1}{2};2k-1;1-\left(\tfrac{r_1-r_2}{r_1+r_2}\right)^2\right).
\end{align*}
}By the duplication formula for the gamma function, 
\begin{align*}
\frac{\sqrt{\pi}\Gamma\left(k-\tfrac{1}{2}\right)}{\Gamma(k)}
B\left(k+\tfrac{z-1}{2},k+\tfrac{-z-1}{2}\right)
&=
\frac{\sqrt{\pi}\Gamma\left(k-\tfrac{1}{2}\right)}{\Gamma(k)}
\times \frac{\Gamma\left(k+\tfrac{z-1}{2}\right)\Gamma\left(k+\tfrac{-z-1}{2} \right)}{2^{2k-2}\sqrt{\pi}^{-1}\Gamma\left(k-\tfrac{1}{2}\right)\Gamma(k)}
\\
&=2^{2-2k}\pi \frac{\Gamma\left(k+\tfrac{z-1}{2}\right)\Gamma\left(k+\tfrac{-z-1}{2}\right)}{\Gamma(k)^2}.
\end{align*}
\end{proof}

\begin{lem}\label{OIestL5} Let $x\in \R$ be such that $0<x<1$ and $z\in i\R$. Then 
\begin{align*}
|{}_2F_1\left(k-\tfrac{1}{2},k+\tfrac{z-1}{2};2k-1;x\right)|
\leq (1-x)^{-1} \frac{1}{k-1/2}\frac{\Gamma(2k-1)}{\Gamma(k-1/2)^2}. 
\end{align*}
\end{lem}
\begin{proof}
From \cite[p.54]{MOS}, 
\begin{align*}
|{}_2F_1\left(k-\tfrac{1}{2},k+\tfrac{z-1}{2};2k-1;x\right)|
&= \left|\tfrac{\Gamma(2k-1)}{\Gamma(k-1/2)^2}\int_{0}^{1}t^{k-3/2}(1-t)^{k-3/2}(1-xt)^{-k-(z-1)/2}\d t\right|
\\
& \leq \tfrac{\Gamma(2k-1)}{\Gamma(k-1/2)^2} \int_{0}^{1}t^{k-3/2}(1-xt)^{-1}\left(\tfrac{1-t}{1-tx}\right)^{k-3/2}\d t. 
\end{align*}
Since $0\le \frac{1-t}{1-tx}<1$ for $0\leq t\leq 1$, we have 
\begin{align*}
\int_{0}^{1}t^{k-3/2}(1-xt)^{-1}\left(\tfrac{1-t}{1-tx}\right)^{k-3/2}\d t
 \leq (1-x)^{-1} \int_{0}^{1}t^{k-3/2}\d t=(1-x)^{-1}\tfrac{1}{k-1/2}. 
\end{align*}
\end{proof}

\begin{lem} \label{OIestL6}
Let $d_1>d_2$ be two positive integers and $z\in i\R$. Then 
\begin{align*}
|\Ocal_{k}^{+,(z)}\left(\tfrac{d_1+d_2}{d_1-d_2}\right)\,\Ocal_k^{-,(z)}(0)^{-1}|
& \ll_{z,d_1,d_2} {k}^{-1/2}\,\left(\tfrac{4\sqrt{d_1d_2}}{(\sqrt{d_1}+\sqrt{d_2})^2} \right)^{k}.
\end{align*}
\end{lem}
\begin{proof} From Lemmas~\ref{OIestL1}, \ref{OIestL4} and \ref{OIestL5}, $|\Ocal_{k}^{+,(z)}\left(\tfrac{d_1+d_2}{d_1-d_2}\right)\,\Ocal_k^{-,(z)}(0)^{-1}|$ is majorized by 
\begin{align*}
&|(d_1-d_2)^2|^{-k/2}(d_1d_2)^{k/2}2^{2k}\times 
2^{-2k}\frac{\Gamma\left(k+\tfrac{z-1}{2}\right)\Gamma\left(k+\tfrac{-z-1}{2}\right)}{\Gamma(k)^2}
\left(\tfrac{d_1^{1/2}+d_2^{1/2}}{d_1^{1/2}-d_2^{1/2}}\right)^{-k}
\,\frac{1}{k-1/2}\frac{\Gamma(2k-1)}{\Gamma(k-1/2)^2}
\\ &\quad 
\times 2^{-k}\Gamma\left(\tfrac{k}{2}+\tfrac{-z-1}{4}\right)^{-1}
\Gamma\left(\tfrac{k}{2}+\tfrac{z-1}{4}\right)^{-1}\Gamma(k)
\end{align*}
up to a constant multiple depending on $z$, $d_1$ and  $d_2$.
By applying the duplication formula to $\Gamma(2k-1)$, $\Gamma\left(k+\tfrac{z-1}{2}\right)$ and $\Gamma\left(k+\tfrac{-z-1}{2}\right)$, this is majorized by 
{\allowdisplaybreaks\begin{align*}
&\left(\tfrac{\sqrt{d_1d_2}}{(\sqrt{d_1}+\sqrt{d_2})^2}\right)^{k} 
\frac{2^{2k-2}\sqrt{\pi}^{-1}\Gamma(k)\Gamma(k-1/2)}{\Gamma(k+1/2)\Gamma(k-1/2)}\\
&\times 
{2^{k+(z-1)/2-1}\sqrt{\pi}^{-1} 
\Gamma\left(\tfrac{k}{2}+\tfrac{z-1}{4}\right)
 \Gamma\left(\tfrac{k}{2}+\tfrac{z+1}{4}\right)\,
2^{k+(-z-1)/2-1}\sqrt{\pi}^{-1}\Gamma\left(\tfrac{k}{2}+\tfrac{-z-1}{4}\right)
\Gamma\left(\tfrac{k}{2}+\tfrac{-z+1}{4}\right)}{\Gamma(k)^{-2}}
\\
&
\times 2^{-k}\Gamma\left(\tfrac{k}{2}+\tfrac{-z-1}{4}\right)^{-1}
\Gamma\left(\tfrac{k}{2}+\tfrac{z-1}{4}\right)^{-1}\Gamma(k) 
\\
&\ll \left(\tfrac{\sqrt{d_1d_2}}{(\sqrt{d_1}+\sqrt{d_2})^2}\right)^{k} \,
2^{3k}\,\Gamma\left(\tfrac{k}{2}+\tfrac{z+1}{4}\right)\Gamma\left(\tfrac{k}{2}+\tfrac{-z+1}{4}\right)\Gamma(k+1/2)^{-1}
\\
&=\left(\tfrac{\sqrt{d_1d_2}}{(\sqrt{d_1}+\sqrt{d_2})^2}\right)^{k} \,
2^{2k+1/2}\sqrt{\pi}\,\frac{\Gamma\left(\tfrac{k}{2}+\tfrac{z+1}{4}\right)}{\Gamma\left(\tfrac{k}{2}+\tfrac{1}{4}\right)}\frac{\Gamma\left(\tfrac{k}{2}+\tfrac{-z+1}{4}\right)}{\Gamma\left(\tfrac{k}{2}+\tfrac{3}{4}\right)} \ll
\left(\tfrac{4\sqrt{d_1d_2}}{(\sqrt{d_1}+\sqrt{d_2})^2}\right)^{k} \tfrac{1}{\sqrt{k}}.
\end{align*}
}Here the last majorization is due to \eqref{asym_gamma}.
\end{proof}

Set $\Tcal(n)=\bigcup_{D \in \Dcal}\Tcal(n,D)$.
From Theorem~\ref{MVT-L3-1},  
{\allowdisplaybreaks\begin{align*}
\tfrac{4\pi}{k-1}n^{1/2}{\bf A}_{k,n}(\phi)
= & \PP_{-4n}(\phi)\,{\bf B}(\nu;-4n)\,\Ocal_{k}^{-,(\nu_\infty)}(0)\\
&+\sum_{\substack{t\in \Tcal(n),\\ t^2<4n}} \PP_{t^2-4n}(\phi)\,{\bf B}(\nu,t^2-4n)\,\Ocal_k^{-,(\nu_\infty)}\left(\tfrac{t}{\sqrt{|t^2-4n|}}\right)\\
&
+\tfrac{1}{4}\hat L\left(\tfrac{1}{2},\phi \right)\sum_{\substack{n=d_1d_2 \\ d_1,d_2>0,d_1\not=d_2}}{\bf B}(\nu;(d_1-d_2)^2) 
\,\Ocal_{k}^{+,(\nu_\infty)}\left(\tfrac{d_1+d_2}{d_1-d_2}\right)\\
& +\tfrac{1}{2}\sum_{\substack{t\in \Tcal(n) \\ t^2>4n}} \PP_{t^2-4n}(\phi)\, {\bf B}(\nu;t^2-4n)\,\Ocal_k^{+,(\nu_\infty)}\left(\tfrac{t}{\sqrt{t^2-4n}}\right).
\end{align*}
}From this and Lemma~\ref{OIestL1}, 
\begin{align*}
&\frac{2^{3-k}i^{-k}\sqrt{\pi}\Gamma(k-1)}{\Gamma\left(\tfrac{k}{2}+\tfrac{-\nu_\infty-1}{4}\right)\Gamma\left(\tfrac{k}{2}+\tfrac{\nu_\infty-1}{4}\right)}
n^{1/2}{\bf A}_{k,n}(\phi)
=\PP_{-4n}(\phi)\,{\bf B}(\nu;-4n)
+E_1(k)+E_2(k)+E_3(k)
\end{align*}
with
{\allowdisplaybreaks\begin{align*}
E_1(k)&=\sum_{\substack{t\in \Z \\ 0<|t|<2\sqrt{n}}} \PP_{t^2-4n}(\phi)\,{\bf B}(\nu;t^2-4n)\,\Ocal_k^{-,(\nu_\infty)}\left(\tfrac{t}{\sqrt{4n-t^2}}\right)\Ocal_{k}^{-,(\nu_\infty)}(0)^{-1}, 
\\
E_2(k)&=\tfrac{1}{2}\sum_{\substack{t\in \Tcal(n) \\ t^2>4n}} 
\PP_{t^2-4n}(\phi)\, {\bf B}(\nu;t^2-4n)\,\Ocal_k^{+,(\nu_\infty)}\left(\tfrac{t}{\sqrt{t^2-4n}}\right)\Ocal_k^{-,(\nu_\infty)}(0)^{-1}, \\
E_3(k)&=\tfrac{1}{4}\hat L\left(\tfrac{1}{2},\phi\right)\sum_{\substack{n=d_1d_2 \\ d_1,d_2>0,d_1\not=d_2}}{\bf B}(\nu;(d_1-d_2)^2) 
\,\Ocal_{k}^{+,(\nu_\infty)}\left(\tfrac{d_1+d_2}{d_1-d_2}\right)
\Ocal_{k}^{-,(\nu_\infty)}(0)^{-1}. 
\end{align*}
}By \eqref{asym_gamma}, the factor in front of $n^{1/2}\bfA_{k,n}(\phi)$ is estimated as
\begin{align*}
\frac{2^{3-k}i^{-k}\sqrt{\pi}\Gamma(k-1)}{\Gamma(\frac{k}{2}+\frac{-\nu_\infty-1}{4})\Gamma(\frac{k}{2}+\frac{\nu_\infty-1}{4})}
&=\frac{2^{3-k}i^{-k}\sqrt{\pi} \times 2^{k-2}\sqrt{\pi}^{-1}\Gamma((k-1)/2)\Gamma(k/2)}{\Gamma(\frac{k}{2}+\frac{-\nu_\infty-1}{4})\Gamma(\frac{k}{2}+\frac{\nu_\infty-1}{4})}
\\
&=2(-1)^{k/2}\left(\tfrac{k}{2}\right)^{-1/2+(\nu_\infty+1)/4} \times \left(\tfrac{k}{2}\right)^{-(\nu_\infty-1)/4}+O(k^{-1})
\\
&=2(-1)^{k/2}+O(k^{-1}).
\end{align*}

Next we consider $E_2(k)$.
From \cite[Lemma 7.14 (1)]{SugiyamaTsuzuki2018}, for any $t,f\in \Z$ such that $|t|>2\sqrt{n}$ and $(t^2-4n)/f^2$ is a fundamental discriminant, we have
\begin{align*}
|{\bf B}(\nu;t^2-4n)|&\leq \prod_{p|f}3|\ord_p(f^{-2})||f^{-2}|_p^{(|\Re \nu_p|+1)/4}\\
&\leq \prod_{p|f}3|\ord_p(f^{-2})||f^{-2}|_p^{2/4} \ll_{\e}(f^2)^{1/2+\e}\leq (t^2-4n)^{1/2+\e}.\end{align*}
Since $\PP_{t^2-4n}(\phi)=2^{-\delta(D<0)}\Pcal_{t^2-4n}(\tilde \phi) \ll_{\e} (t^2-4n)^{1/2+\e}$ for $|t|> 2\sqrt{n}$ follows from the proof of \cite[Theorem 9.1]{SugiyamaTsuzuki2018}, by Lemmas~\ref{OIestL1} and \ref{OIestL3}, we have
{\allowdisplaybreaks\begin{align*}
|E_2(k)|\ll_{\e}
& \sum_{\substack{t\in \Tcal(n) \\ t^2>4n}}
|\PP_{t^2-4n}(\phi)|\times |{\bf B}(\nu;t^2-4n)|\times 
 \{|I_k\Bigl(4,\tfrac{2t}{\sqrt{t^2-4n}};\tfrac{\nu_\infty+1}{2}\Bigr)|
+|I_k\Bigl(4,\tfrac{-2t}{\sqrt{t^2-4n}};\tfrac{\nu_\infty+1}{2}\Bigr)|\}
\\
&\qquad \times 2^{2k}|\Delta|^{-k/2}n^{k/2}\times 2^{-k}|\Gamma\left(\tfrac{k}{2}+\tfrac{-\nu_\infty-1}{4}\right)
\Gamma\left(\tfrac{k}{2}+\tfrac{\nu_\infty-1}{4}\right)\Gamma(k)^{-1}|^{-1}
\\
\ll_{\e} &
\sum_{|t|>2\sqrt{n}}
(t^2-4n)^{1/2+\e}\times (t^2-4n)^{1/2+\e}\times (t^2-4n)^{-k/2}2^{k}n^{k/2}\sqrt{k}\left(\tfrac{2|t|}{\sqrt{t^2-4n}}\right)^{-k+\e+1/2}
\\
\ll_{\e}& \sum_{|t|>2\sqrt{n}} (t^2-4n)^{3/4+3\e/2}\,\left(\tfrac{|t|}{\sqrt{n}}\right)^{-k+1/2+\e}\,\sqrt{k}.
\end{align*}
}Let $t_0$ be the least positive integer such that $t_0>2\sqrt{n}$. With a fixed $k_0>2$ and a sufficiently small $\e>0$, the last quantity is majorized by  
\begin{align*}
\sqrt{k}\left(\tfrac{t_0}{\sqrt{n}}\right)^{-(k-k_0)} \times \sum_{|t|>2\sqrt{n}}(t^2-4n)^{3/4+3\e/2}\,\left(\tfrac{|t|}{\sqrt{n}}\right)^{-k_0+\e+1/2} \ll \sqrt{k} \left(\tfrac{t_0}{\sqrt{n}}\right)^{-k}.
\end{align*}
Hence $E_2(k)=O(\sqrt{k}(t_0/\sqrt{n})^{-k})$, which shows $\lim_{k\rightarrow \infty}E_2(k)=0$. By Lemma~\ref{OIestL6}, 
\begin{align*}
|E_3(k)|\ll k^{-1/2} \sum_{\substack{d_1,d_2\in \N \\ n=d_1d_2 \\d_1\neq d_2}}|{\bf B}(\nu;(d_1-d_2)^{2})|\,C_{\nu_\infty, d_1,d_2} \left(\tfrac{4\sqrt{d_1d_2}}{(\sqrt{d_1}+\sqrt{d_2})^2} \right)^{k},
\end{align*}
where $C_{\nu_\infty, d_1, d_2}>0$ is independent of $k$.
Since the pairs $(d_1, d_2)$ are finite and we have the strict inequality $4\sqrt{d_1d_2}< (\sqrt{d_1}+\sqrt{d_2})^2$ which ensures the exponential decay, we see $E_3(k)=O(k^{-1})$.
The term $E_1(k)$ is evaluated by Proposition \ref{2017/12/04-3}.
Therefore, Theorem~\ref{ASYMPT} follows from the considerations so far. To deduce Theorem~\ref{ASYMPT_average},
we consider the average of $E_1(k)$ over $k$:
\begin{align*}
\tilde E_1({\rm K})=\frac{1}{2^{-1}{\rm K}}\sum_{ k\in 2\NN \cap[{\rm K},2{\rm K})}E_{1}(k).
\end{align*}
Here ${\rm K}$ is a positive even integer.
For each $(t,n)$ such that $0<|t|<2\sqrt{n}$, let $\theta_{t,n}\in [0,2\pi)$ be the argument of the complex number $\rho=2^{-1}(-t+i\sqrt{|\Delta|})$. By $t\not=0$
and $\Delta\not=0$, we have $2\theta_{t,n} \notin 2\pi\Z$. Hence
$$
\frac{1}{{\rm K}}\sum_{k\in 2\NN \cap [{\rm K},2{\rm K})}e^{\pm i\theta_{t,n}k}=\frac{1}{{\rm K}}\sum_{\substack{m\in \N \\ {\rm K}/2\le m<{\rm K}}}e^{\pm2i\theta_{t,n}m}= \frac{1}{{\rm K}}\frac{e^{\pm{\rm K}i\theta_{t,n}} 
	-e^{\pm2{\rm K}i\theta_{t,n}}}{1-e^{\pm2i\theta_{t,n}}}=O({\rm K}^{-1}). 
$$
This combined with Proposition~\ref{2017/12/04-3} yields
{\allowdisplaybreaks\begin{align*}
|\tilde E_{1}({\rm K})|&\ll \sum_{\substack{t\in \Z \\ 0<|t|<2\sqrt{n}}}
\left|\frac{2}{{\rm K}}\sum_{k\in 2\NN \cap [{\rm K},2{\rm K})}
\Ocal_{k}^{-,(\nu_\infty)}\left(\tfrac{t}{\sqrt{|\Delta|}}\right)/\Ocal_{k}^{-,(\nu_\infty)}(0)\right|
\\
& \ll\sum_{\substack{t\in \Z \\ 0<|t|<2\sqrt{n}}} \left|\frac{1}{{\rm K}}\sum_{k\in 2\NN \cap [{\rm K},2{\rm K})}(\rho e^{-i\theta_{t,n}k}-\bar\rho e^{i\theta_{t,n}k})\right|
+\left|\frac{1}{{\rm K}}\sum_{{\rm K}\le k<2{\rm K}}\frac{1}{k} \right| 
=O({\rm K}^{-1}). 
\end{align*}}This completes the proof of Theorem~\ref{ASYMPT_average}.

\section{Quantitative non-vanishing}

In this section, we give a proof of Theorem~\ref{QuaNonVan}. We rely on Theorem ~\ref{ASYMPT_average} as well as on the asymptotic formula \eqref{LuoSar3} which is proved by \cite{LuoSarnak} with an additional argument as in \cite[\S5]{SarnakZhaoWoo} to remove the weight $L(1,f,{\rm sym}^2)$. The convergence of the Euler product $C(\phi)$ follows from the Kim-Sarnak bound \cite[Appendix 2]{Kim}; note that $C(\phi)>0$. Recall our counting function $N_{\phi,n}({\rm K})$. From Watson's formula (\cite{Watson}; see also \cite[p.785]{LuoSarnak}), $L(1/2,\phi)\,L(1/2,\phi\times {\rm Ad}\,f)\not=0$ if and only if $\mu_{f}(\phi)\not=0$. Hence
$$
N_{\phi, n}({\rm K})=\#\biggl\{f \in \bigcup_{{\rm K}\le k <2{\rm K}} H_k \ \bigg| \ \mu_f(\phi)\lambda_f(n)\neq 0 \biggr\} \quad \text{if $L(1/2,\phi)\neq 0$}. 
$$
Let $Y(\phi)$ be the set of $n\in \N$ such that ${\bf B}(\nu;-4n)\,\PP_{-4n}(\phi)\not=0$. 

\begin{prop}\label{quantitative version} 
Let $\phi$ be an even Hecke-Maass cusp form on $\SL_2(\ZZ)$ with $\lambda_\phi(1)=1$ and $L(1/2, \phi)\neq 0$. For any $n\in Y(\phi)$ and for any $\e>0$, there exists ${\rm K}(\phi,n,\e) \in \NN$ such that 
$$
\frac{N_{\phi,n}({\rm K})}{{\rm K}} \ge \frac{(1-\e)^2}{8\pi(1+\e)}\frac{|\PP_{-4n}(\phi)\bfB(\nu;-4n)|^2}{C(\phi)L(1/2,\phi)\|\phi\|^2
\, n\, d(n)^2} \quad \text{for all ${\rm K} \ge {\rm K}({\phi,n,\e})$}.$$
\end{prop}
\begin{proof} From Theorem \ref{ASYMPT_average}, for any $n\in Y(\phi)$ and $\e>0$, there exists ${\rm K}_1({\phi,n,\e}) \in \NN$ such that for any ${\rm K}\ge {\rm K}_1({\phi,n,\e})$,
\begin{align*}
& (1-\e)\left|\frac{1}{\sqrt{4n}}\PP_{-4n}(\phi)\bfB(\nu;-4n)\right| \le 
 \frac{2}{{\rm K}}\left| \sum_{{\rm K}\le k<2{\rm K}}\sum_{f \in H_k}{(-1)^{k/2}\mu_f(\phi)}\,\lambda_{f}(n) \right|.
\end{align*}
By the Cauchy-Schwarz inequality and by Deligne's bound $|\lambda_f(n)|\leq d(n)$, the quantity on the right-hand side is no greater than 
\begin{align*}
&\frac{2}{{\rm K}}\biggl\{\sum_{{\rm K}\le k< 2{\rm K}}\sum_{f\in H_k} {|\mu_f(\phi)|^2} \biggr\}^{1/2} \biggl\{\sum_{{\rm K}\le k < 2{\rm K}}\sum_{f \in H_k; \mu_{f}(\phi) \lambda_{f}(n) \neq 0} 1 \biggr\}^{1/2}d(n).
\end{align*}
Since $L(1/2,\phi)>0$ from the assumption in conjunction with \cite[Corollary 1]{KatokSarnak}, the asymptotic formula \eqref{LuoSar3} yields a constant ${\rm K}_2(\phi,\e)>0$ such that the inequality
\begin{align*}\sum_{{\rm K}\le k< 2{\rm K}}
\sum_{f\in H_k} {|\mu_f(\phi)|^2}\leq 
(1+\e)\frac{\pi {\rm K}}{2}C(\phi) L(1/2, \phi)\|\phi\|^2 
\end{align*}
holds for all ${\rm K}\geq {\rm K}_2(\phi,\e)$. Set ${\rm K}(\phi,n,\e)=\max\{{\rm K}_1(\phi,n,\e),{\rm K}_2(\phi,\e)\}$. Then for all ${\rm K}\geq {\rm K}(\phi,n,\e)$, we end up with the desired equality in the form  
\begin{align*}
&(1-\e)\left|\frac{1}{\sqrt{4n}}\PP_{-4n}(\phi)\bfB(\nu;-4n)\right| \leq \frac{2}{{\rm K}}\,d(n)
\biggl\{(1+\e)\frac{\pi {\rm K}}{2}C(\phi) L(1/2, \phi)\|\phi\|^2 \biggr\}^{1/2}\,N_{\phi,n}({\rm K})^{1/2}. 
\end{align*}
\end{proof}

To have a smooth access to a result of \cite{MartinWhitehouse}, we bridge the two different notations of automorphic $L$-functions. The standard $L$-function of any irreducible cuspidal representation $\pi$ of $\GL_n(\A_\QQ)$ ($n=2,3$) is denoted by $L(s,\pi)$ (completed by gamma factors); the Euler product without gamma factors is denoted by $L_{\fin}(s,\pi)$. For $D\in \Dcal$, $\varepsilon_{D}$ denotes the idele class character of $\Q^\times$ induced by $\chi_{D}$. Let $\pi_{\phi}$ be the cuspidal representation generated by $\tilde \phi$ (see the first paragraph of \S~\ref{Proof of Theorem MVT-L3-1}). Then $L_\fin(s,\pi_\phi)=L(s,\phi)$, $L_\fin(s,\pi_\phi\otimes \varepsilon_D)=L(s,\phi\otimes \chi_{D})$ and $L_\fin(s,{\Ad}(\pi_\phi))
=L(s,\phi,{\rm sym}^2)$.

\begin{lem} \label{WaldspurgerFormulaL}
 Let $\phi$ be an even Hecke-Maass cusp form on $\SL_2(\Z)$ with $\lambda_\phi(1)=1$, and let $D<0$ be a fundamental discriminant. Then, we have
$$\frac{|\PP_{D}(\phi)|^2}{\|\phi\|^2} = \frac{\sqrt{|D|}}{4}
\frac{L_{\fin}(1/2,\pi_\phi)\,L_{\fin}(1/2,\pi_\phi \otimes \varepsilon_{D})}{L_\fin(1, {\Ad}(\pi_{\phi}))}.
$$
\end{lem}
\begin{proof} Let $\tilde \phi$ be the function on $\GL_2(\A_\Q)$ defined in \S~2. We apply \cite[Theorem 4.1]{MartinWhitehouse} taking $F=\Q$, $E=\Q(\sqrt{D})$, $\Omega={\bf 1}$ and $\pi=\pi_{\phi}$. Let $\d \tau$ be an additive Haar measure on $\A_{E}$ which is self-dual with respect to the self-duality defined by the character $\psi_{E}=\psi\circ {\rm tr}_{E/\Q}$. Then $\d\tau=\otimes_{v}\d\tau_v$ with $\d \tau_v$ the self-dual Haar measure on $E_v=E\otimes_{\Q}\Q_v$ with respect to $\psi_{E_v}=\psi_{E}|E_v$. Then $\vol(\cO_{E_p},\d\tau_p)=p^{-\ord_{p}D/2}$ for all $p<\infty$ and $\cO_{E_p}$ being the maximal compact subring of $E_p$. Recall the optimal embedding $\Psi_{D}:E\hookrightarrow \Mat_{2}(\Q)$ of level $1$ defined in \S~\ref{sect:PeriodInt}; then the order $R(\pi_p)$ in \cite[\S 4.1]{MartinWhitehouse} coincides with $\Mat_{2}(\Z_p)$.  When $D\equiv 1 \pmod{4}$, the formula \eqref{PsiKinfty} shows that $R(\gamma_{D,\infty})\tilde \phi$ (the right translation of $\tilde \phi$ by $\gamma_{D,\infty}$) is right $\Psi_{D}^{-1}(\SO(2))$-invariant, and this $\Psi_{D}^{-1}(\SO(2))$-invariance holds also in the case $D\equiv 0 \pmod{4}$ similarly. Thus the vector $R(\gamma_{D,\infty})\tilde \phi$ satisfies the required local conditions for \cite[Theorem 4.1]{MartinWhitehouse} to be applied. Hence
\begin{align}
\left|\int_{\A_\Q^{\times} E^\times \bsl \A_E^\times}\tilde \phi (\Psi_{D}(\tau)\gamma_{D,\infty})\,\d^*\dot \tau\right|^2\,(\tilde \phi,\tilde \phi )^{-1}=\frac{L(1/2,\pi_{\phi})\,L(1/2,\pi_{\phi}\otimes \varepsilon_{D})}{L(1,{\Ad}(\pi_{\phi}))}\,\frac{1}{2\sqrt{|D|}}\,C_{\infty}(E,\pi_{\phi},{\bf 1}),
 \label{WaldspurgerFormulaL-1}
\end{align}
where $\d^* \dot \tau$ is the quotient measure on $\A_\Q^\times E^\times \bsl \A_E^\times$ induced from the Haar measure $\d^*\tau=\otimes_{v}\zeta_{E_v}(1)\frac{\d\tau_v}{|\nr_{E/\Q}(\tau_v)|_{v}}$ on $\A_E^\times$ and the Haar measure $\d^\times t=\otimes_{v}\zeta_v(1)\frac{\d \lambda_v(t_{v})}{|t_v|_{v}}$ on $\A_\Q^\times$, $(\,,\,)$ is the $L^2$-inner product on $L^2(Z(\A_\Q) \GL_2(\Q)\bsl \GL_2(\A_\Q))$ with $Z$ being the center of $\GL_2$ and $C_{\infty}(E,\pi_{\phi},{\bf 1})$ is the constant evaluated to be $1$ in \cite[p.175, line 7--12]{MartinWhitehouse}. We remark that the Haar measure $d^*\tau$ on $\A_{E}^\times$
is 
the same as
our $\d^\times \tau=\otimes_{v} \zeta_{E_v}(1) \frac{\d\lambda_{v}(x)\d\lambda_v(y_v)}{|x_v^2-4^{-1}D y_v^2|_v}$ (see \S~\ref{sect:PeriodInt}).
Indeed, by $\vol(\cO_{E,p}^\times; \d^* \tau_p)=p^{-\ord_p(D)/2}$ for $p<\infty$, from Lemma~\ref{PeriodL3}, we have $\d^\times \tau_\fin=|D|^{1/2}\,\d^* \tau_{\fin}$ for the finite component. By a change of variables, it is immediately seen that $\d^\times \tau_\infty=|D|^{-1/2}\d^* \tau_\infty$ for the archimedean component. Thus $\d^\times \tau = \d^* \tau$. Having this adjustment of Haar measures, from the proof of Proposition~\ref{PeriodL0} (especially, the argument after \eqref{PsiKinfty}), we see 
\begin{align*}
\int_{\A_\Q^\times E^\times \bsl \A_E^\times}
\tilde \phi (\Psi_{D}(\tau)\gamma_{D,\infty})\,\d^*\dot \tau
&= \int_{\A_\Q^\times E^\times \bsl \A_E^\times}
\tilde \phi (\Psi_{D}(\tau)\gamma_{D,\infty})\,\d^\times \dot \tau
\\
&= |D|^{-1/2}\Pcal_D(\tilde\phi) = 2\,|D|^{-1/2}\PP_{D}(\phi). 
\end{align*} 
By comparing the normalization of Haar measures on $\GL_2(\A_\Q)$, it is comfirmed that $(\tilde \phi,\tilde\phi)=\|\phi\|^2$. Since $D<0$, the gamma factor of $L(s,\pi_\phi\otimes \varepsilon_{D})$ is $\Gamma_{\R}(s+\nu_\infty/2+1)\Gamma_{\R}(s-\nu_\infty/2+1)$. Hence on the right-hand side of \eqref{WaldspurgerFormulaL-1}, the archimedean $L$-factor are computed as
\begin{align*}
\frac{L_\infty(1/2, \pi_\phi)L_\infty(1/2,\pi_\phi \otimes \varepsilon_{D})}{L_\infty(1,{\Ad}(\pi_\phi))}
= \frac{\Gamma_\RR(\frac{1+\nu_\infty}{2})\Gamma_\RR(\frac{1-\nu_\infty}{2})\Gamma_\RR(\frac{3+\nu_\infty}{2})\Gamma_\RR(\frac{3-\nu_\infty}{2})}{\Gamma_\RR(1+\nu_\infty)\Gamma_\RR(1-\nu_\infty)\Gamma_\RR(1)} = 2.
\end{align*}
Thus we obtain the required formula from \eqref{WaldspurgerFormulaL-1}.
\end{proof}


\noindent
{\it Proof of Theorem~\ref{QuaNonVan}} :
From the formula \eqref{bfB-defformula}, we have ${\bf B}(\nu;-4n)=1$ for any $n\in X(\phi)$. By Lemma~\ref{WaldspurgerFormulaL}, we have $X(\phi)\subset Y(\phi)$.  Thus, Theorem~\ref{QuaNonVan} follows from Proposition~\ref{quantitative version} and Lemma~\ref{WaldspurgerFormulaL} immediately.

\section{Appendix: explicit formulas of dyadic elliptic orbital integrals}
\label{Appendix}


The aim of the appendix is to complete \cite[Theorem 7.9 (5)]{SugiyamaTsuzuki2018} by proving an explicit formula of the integral
$$\fE^{(z)}(\hat\gamma)=\int_{{\mathfrak T}_\tau \bsl \GL_2(\QQ_2)}\Phi(s; g^{-1}\hat\gamma g)\varphi^{\tau}(g)dg,
$$
the definition of whose ingredients is recalled in what follows. Let $|\cdot|$ be the $2$-adic valuation of $\QQ_2$ such that $|2|=2^{-1}$. Let $t,n$ be elements of $\QQ_2$ with $t^2-4n\not=0$, which are fixed throughout the appendix. Then, there exist $m \in \QQ_2^\times$ and $\tau \in \{1\}\cup2\ZZ_2^\times\cup\{\ZZ_2^\times-(\ZZ_2^\times)^2\}$
such that $t^2-4n=(2m)^2\tau$.
Set $E=\QQ_2[X]/(X^2-\tau)$.
If $\tau \neq 1$, then $E/\QQ_2$ is a quadratic extension and $E$ is embedded into ${\rm M}_2(\QQ_2)$ by
$\iota_\tau(a+b\sqrt{\tau})=[\begin{smallmatrix}
a & b \\ b\tau & a
\end{smallmatrix}]$ for any $a,b \in \QQ_2$.
Remark that all the quadratic extensions of $\QQ_2$ are exhausted by $\QQ_2[\sqrt{\tau}]$
with $\tau \in \{-1, \pm 2, \pm 5, \pm 10\}$.
If $\tau=1$, then $E$ is isomorphic to $\QQ_2\times \QQ_2$ and embedded into ${\rm M}_2(\QQ_2)$ by
$\iota_\tau(a+b\sqrt{\tau})=[\begin{smallmatrix}
a+b & 0 \\ 0 & a-b
\end{smallmatrix}]$ for any $a,b \in \QQ_2$. The subgroup $\mathfrak{T}_{\tau}:=\iota_\tau(E^\times)$ is an elliptic torus of $\GL_2(\Q_2)$ unless $\tau=1$, in which case it is a split torus. Associated with $(t,n)$, an element $\hat \gamma \in \mathfrak{T}_\tau$ is defined as
$\hat\gamma=[\begin{smallmatrix}
a & 1 \\ \tau & a
\end{smallmatrix}]$ if $\tau\neq1$, and
$\hat\gamma=[\begin{smallmatrix}
a+1 & 0 \\ 0 & a-1
\end{smallmatrix}]$ if $\tau=1$, where $a=\frac{t}{2m}$.

Set $G=\GL_2(\QQ_2)$.
Let $\Phi(s) : G\rightarrow \CC$ be the $2$-adic Green function defined by
$$\Phi(s;g)=(2^{-\frac{s+1}{2}}-2^{\frac{s+1}{2}})^{-1}\{|\det g|^{-1}\max_{1\le i,j \le 2}(|g_{ij}|)^{2} \}^{-\frac{s+1}{2}}, \qquad g=(g_{ij}) \in G$$
for $\Re s>1$ (cf.\ \cite[\S2.3]{SugiyamaTsuzuki2018}). We fix a Haar measure $dg$ on $G$ such that $\vol(\GL_2(\ZZ_2),dg)=1$.
We put $d^\times x=\zeta_2(1)|x|^{-1}d\l_2(x)$ on $\QQ_2^\times$ and define a Haar measure on $E^\times$ as
$$d^\times u=\zeta_{E}(1)|x^2-m^2\tau y^2|^{-1}dxdy \text{ with } u=x+\sqrt{m^2\tau}\,y.
$$
By ${\frak T}_\tau\cong E^\times$, we regard $d^\times u$ as a Haar measure on $\mathfrak{T}_\tau$ and endow the space $\mathfrak{T}_\tau \bsl G$ with the quotient measure $dg/d^\times u$.

Recall that there exists a unique function $\varphi^\tau:G\rightarrow \C$ which satisfies the conditions (i) $\varphi^\tau*{\mathbb T}_2=(2^{(1+z)/2}+2^{(1-z)/2})\varphi^\tau$, (ii) $\varphi^\tau(tg)=\varphi^\tau(g)\,(t\in {\frak T}_\tau,\,g\in G)$, and (iii) $\varphi^\tau(1_2)=1$
(cf.\ \cite[Lemma 7.5]{SugiyamaTsuzuki2018}), where ${\mathbb T}_2$ is the characteristic function of the set $\{g\in {\rm M}_2(\ZZ_2) \mid |\det g|=2^{-1}\}$. If $\tau=1$, the formula of $\varphi^\tau$ is given in \cite[Lemma 6.1]{SugiyamaTsuzuki2018}. If $\tau\not=1$, the function $\varphi^\tau$ is constructed as follows. Let $f$ be a spherical vector in ${\rm Ind}_{B(\QQ_2)}^{G}(|\cdot|^{z/2}\boxtimes|\cdot|^{-z/2})$ determined by $f(1_2)=1$ and we define $\varphi_0: G\rightarrow \CC$ by
$$\varphi_{0}(g)=\int_{Z\bsl \mathfrak{T}_\Delta}f(tg)dt, \qquad g\in G,$$
where $Z$ is the center of $G$. We need the value $\varphi_0(1_2)$ which is computed in	\cite[Lemma 7.4]{SugiyamaTsuzuki2018} as
\begin{align}\label{phi(1)}\varphi_0(1_2)=\begin{cases}
|m|^{-1}(1+2^{\frac{z+1}{2}}) & (\tau \in \{\pm2, \pm10\}), \\
|m|^{-1}(1+2^{-\frac{z+1}{2}}) & (\tau \in \{-1,-5 \}), \\
|m|^{-1}3^{-1}2(1+2^{-z}) & (\tau =5).
\end{cases}
\end{align}
Then $\varphi^\tau(g)=\varphi_0(g)/\varphi_0(1_2)$.

For $s,z \in \CC$ such that $\Re s> (|\Re z|-1)/2$ and $b \in \QQ_2^\times$, set
$$\Scal_2^{\tau, (z)}(s;b)=-2^{-\frac{s+1}{2}}\frac{\zeta_2(s+\frac{z+1}{2})\zeta_2(s+\frac{-z+1}{2})}{L(s+1, \varepsilon_\tau)}|b|^{\frac{s+1}{2}}, \qquad (|b|\le 1),$$
$$\Scal_2^{\tau, (z)}(s;b) = -2^{-\frac{s+1}{2}}\left\{ \frac{\zeta_2(-z)\zeta_2(s+\frac{z+1}{2})}{L(\frac{-z+1}{2},\varepsilon_\tau)}|b|^{\frac{-z+1}{4}} + \frac{\zeta_2(-z)\zeta_2(s+\frac{z+1}{2})}{L(\frac{-z+1}{2}, \varepsilon_\tau)} |b|^{\frac{z+1}{4}}\right\}, \qquad (|b|>1)$$
as in \eqref{Scaldeltazsa-f1}, where $\varepsilon_\tau$ is the real-valued character of $\QQ_2^\times$ corresponding to $\QQ_2(\sqrt{\tau})/\QQ_2$ by local class field theory. Here is the main result of this section:
\begin{thm}\label{explicit dyadic} Suppose $\Re s>(|\Re z|-1)/2$. We have 
	$$\fE^{(z)}(\hat{\gamma})=\begin{cases}|2m| \Scal_2^{1,(z)}(s;\tfrac{n}{4m^2}) & (\tau =1),\\
	|m|\Scal_2^{\tau, (z)}(s;\tfrac{n}{m^2}) & (\tau \in \{\pm2, \pm10, -1, -5\}),\\
	|2m| 2^{\frac{-z-1}{2}}3(1+2^{-z})^{-1}\Scal_2^{\tau,(z)}(s;\tfrac{n}{4m^2}) & (\tau=5).
	\end{cases}$$
\end{thm}

\subsection{Proof of Theorem \ref{explicit dyadic}}
For proving Theorem \ref{explicit dyadic},
we must notice the following deduced from Hensel's lemma for $\QQ_2$.
\begin{lem}	\label{dyadic hensel}
	If $\tau \in \{\pm2, \pm10\}$, $|\tau-u^2| = 1$ for all $u \in \ZZ_2^\times$.
	
	If $\tau\in \{-1,-5\}$, $|\tau-u^2| = |2|=2^{-1}$ for all $u \in \ZZ_2^\times$.
	
	If $\tau =5$, $|\tau-u^2| = |4|=2^{-2}$ for all $u \in \ZZ_2^\times$.
\end{lem}
\subsubsection{Split case}
The case $\tau=1$ is computed in the same way as \cite[Theorem 7.9]{SugiyamaTsuzuki2018}, and hence we have
\begin{align*}
\fE^{(z)}(\hat\gamma) = |2m| \fF_2^{(z)}(s;\tfrac{a+1}{a-1}),
\end{align*}
where $\fF_2^{(z)}(s)$ is the function computed in \cite[\S9.1.4]{SugiyamaTsuzuki2018} and $a=\frac{t}{2m} \in \QQ_2-\{1\}$. Set $b=\tfrac{a+1}{a-1}$.
Then, we obtain Theorem \ref{explicit dyadic} for $\tau=1$ by noting
$$\tfrac{b}{(b-1)^2}=\tfrac{a+1}{a-1}(\tfrac{a+1}{a-1}-1)^{-2}=\tfrac{a^2-1}{4}=\tfrac{n}{4m^2}.$$

\subsection{Ramified case}
Set $a=\frac{t}{2m}$ and suppose $a\neq 0$. Put $b=a^{-1}$.
We start from the expression
\begin{align}\label{expression of integral}
	&\varphi_{0}(1_2)\fE^{(z)}(\hat\gamma)=\int_{\mathfrak{T}_{\tau}\bsl G}\Phi(s;g^{-1}\widehat{\gamma} g)
	\varphi_{0}(g)dg \notag \\
	= & |1-b^2\tau|^{\frac{s+1}{2}}|\tau b|^{-\frac{z+1}{2}} (2^{-\frac{s+1}{2}}-2^{\frac{s+1}{2}})^{-1}\int_{\QQ_2^\times}B(t)|t|^{\frac{z-1}{2}}d^\times t,
\end{align}
where
$ B(t)=\int_{\QQ_2}\max(|1-x|,  |t^{-1}(\tau b^{2}-x^2)|, | t|, |1+ x|)^{-s-1}dx$ (cf.\ \cite[\S9.2.4]{SugiyamaTsuzuki2018}).

First let us consider the case where $\QQ_2(\sqrt{\tau})/\QQ_2$ is a ramified quadratic extension.
If $\tau \in 2\ZZ_2^\times$, the same argument as in \cite[\S9.2.4]{SugiyamaTsuzuki2018} works without modification, and we have Theorem \ref{explicit dyadic} for $\tau\in\{\pm2,\pm10\}$.
In what follows, we consider the case $\tau \in \{-1,-5\}$, which is computed similarly to \cite[Lemmas 9.10 and 9.11]{SugiyamaTsuzuki2018}. We must modify arguments using Hensel's lemma in \cite[Lemmas 9.10 and 9.11]{SugiyamaTsuzuki2018} with the aid of Lemma \ref{dyadic hensel}.
Set $[2^{-l/2}]=2^{\lfloor -l/2\rfloor}$ for $l \in \ZZ$.
\begin{lem}\label{ramified1}
	Suppose $|a|>1$ and $\Re s>-1/2$.
	
	When $|t|< |b^2|$,
		$B(t) =|t|^{s+1}|b|^{-2s-1}(2^{-1}+2^s-2^{-s-1})(1-2^{-2s-1})^{-1}.$
	
	When $|b^2|\le|t|\le 1$,
		$B(t)=2^{-1}|t|^{s+1}[|t|^{1/2}]^{-2s-1}2^{-2s-1}(1-2^{-2s-1})^{-1}+[|t|^{1/2}].$

	When $|t|>1$, $B(t)=|t|^{-s}(1-2^{-2s-2})(1-2^{-2s-1})^{-1}$.
\end{lem}
\begin{proof}Let $f(x)$ be the integrand of $B(t)$, and set $D_1=\{x \in \QQ_2 \mid |x|\le \max(|b|,|t|) \}$ and $D_2=\QQ_2-D_1$.
	We divide $B(t)$ into the sum $B_1(t)+B_2(t)$, where $B_j(t)=\int_{D_j}f(x)dx$.

	(i) When $|t|\le |b|$, by noting Lemma \ref{dyadic hensel}, for any $x \in D_1$,
	\begin{align*}
		f(x) =\begin{cases} \max(1,|t^{-1}b^2|)^{-s-1} & (|x|<|b|), \\
			\max(1,|2t^{-1}b^2|)^{-s-1} & (|x|=|b|).
		\end{cases}
	\end{align*}
	Hence
	\begin{align*}
		B_1(t)=\begin{cases}|b| & (|t^{-1}b^2|\le 1),\\
			(2^{-1}+2^{s})|t|^{s+1}|b|^{-2s-1} & (|t^{-1}b^2|>1).
		\end{cases}
	\end{align*}
	The formula of $B_2(t)$ is the same as in the proof of \cite[Lemma 9.10]{SugiyamaTsuzuki2018}.
	
	(ii) When $|t|>|b|$, formulas of $B_1(t)$ and of $B_2(t)$ are the same as in the proof of \cite[Lemma 9.10]{SugiyamaTsuzuki2018}.
	
	Combining these, we have the assertion.
\end{proof}

\begin{lem}\label{ramified2}
	Suppose $|a|\le 1$ and $\Re s>-1/2$.
	
	When $|t|< |b|$,
	$B(t)=|t|^{s+1}|b|^{-2s-1}(2^{-1}+2^s-2^{-s-1})(1-2^{-2s-1})^{-1}$.
	
	When $|b|\le |t|$, $B(t) = |t|^{-s}(1-2^{-2s-2})(1-2^{-2s-1})^{-1}.$
	
\end{lem}
\begin{proof}
	The formula of $B_2(t)$ is same as in the proof of \cite[Lemma 9.11]{SugiyamaTsuzuki2018}. Consider $B_1(t)$.
	
	(i) Suppose $|t| < |b|$. Then, $B_1(t) = |t|^{s+1} |b|^{-2s-1}(2^{-1}+2^s).$
		
	(ii) Suppose $|t|\ge |b|$. Then, $B_1(t) = |t|^{-s}$.

Combining these, we have the assertion. 
\end{proof}

From Lemmas \ref{ramified1} and \ref{ramified2},
we have the following.
\begin{lem}Suppose $\Re s>(|\Re z|-1)/2$. When $|a|\le 1$,
	\begin{align*}\int_{\QQ_2^\times}
		B(t)|t|^{\frac{z-1}{2}}d^\times t=|b|^{-s+\frac{z-1}{2}}
		(1+2^{\frac{-z-1}{2}})
		\frac{\zeta_{2}(s+\frac{z+1}{2})\zeta_{2}(s+\frac{-z-1}{2})}{\zeta_{2}(s+1)}.
	\end{align*}
	When $|a|>1$,
	\begin{align*}\int_{\QQ_2^\times}
		B(t)|t|^{\frac{z-1}{2}}d^\times t=\frac{1+2^{\frac{-z-1}{2}}}{\zeta_{2}(s+1)}\{\zeta_2(z)\zeta_{2}(s+\tfrac{-z+1}{2})+\zeta_2(-z)\zeta_2(s+\tfrac{z+1}{2})|b|^z\}.
	\end{align*}
\end{lem}

By noting $L(s,\varepsilon_{\tau})=1$ and \eqref{phi(1)}, we have Theorem \ref{explicit dyadic} for $\tau \in \{-1, -5\}$ when $a\neq0$.
The case $a=\frac{t}{2m}=0$ is treated in the same way as the discussion at the end of \cite[\S9.2.4]{SugiyamaTsuzuki2018}. Hence we obtain the desired formula.

\subsection{Unramified case}
Finally we consider the case where $\QQ_2(\sqrt{\tau})/\QQ_2$ is an unramified quadratic extension. We may put $\tau =5$. 
The same argument as in \cite[Lemmas 9.10, 9.11]{SugiyamaTsuzuki2018} works with minor modification.
Set $a=\frac{t}{2m}$ and suppose $a\neq0$. Put $b=a^{-1}$.
We start from the expression \eqref{expression of integral}.
and use $f(x)$, $D_1$, $D_2$, $B_1(t)$ and $B_2(t)$ introduced in the proof of Lemma \ref{ramified1}. 
\begin{lem}\label{unram1}
	Suppose $|a|> 1$ and set $b=a^{-1}$. We have
	
	When $|t|<|b^2|$,
	\begin{align*}
		B(t)= \left(|t|^{s+1}|b|^{-2s-1}\frac{1-2^{-2s-2}}{1-2^{-2s-1}}-\frac{1}{2}|t|^{s+1}|b|^{-2s-1}+\frac{1}{2}|b|\right).
	\end{align*}
	
	When $|b^2|\le |t| \le 1$, 
	\begin{align*}
		B(t)=\{2^{-1}|t|^{s+1}[|t|^{1/2}]^{-2s-1}2^{-2s-1}(1-2^{-2s-1})^{-1}+[|t|^{1/2}] \}.
	\end{align*}
	
	When $|t|>1$,
	\begin{align*}
		B(t)= |t|^{-s}(1-2^{-2s-2})(1-2^{-2s-1})^{-1}.
	\end{align*}

\end{lem}
\begin{proof}
	(i) Suppose $|t|\le |b|$.
	For $x \in D_1$, $f(x)=\max(1,|t^{-1}b^2|)^{-s-1}$ if $|x|<|b|$,
	and $f(x)=1$ if $|x|=|b|$.
	Thus \begin{align*}
		B_1(t)= |2b|\max(1,|t^{-1}b^2|)^{-s-1}+2^{-1}|b|.
	\end{align*}
	The same formula of $B_2(t)$ as \cite[Lemma 9.10]{SugiyamaTsuzuki2018} holds.
	
	(ii) Suppose $|t|>|b|$.
	Then, $B_1(t)$ and $B_2(t)$ are computed as in the proof of \cite[Lemma 9.10]{SugiyamaTsuzuki2018},
	and the same formulas of $B_1(t)$ and of $B_2(t)$ hold.
	
	Combining these, we have the assertion.
\end{proof}

\begin{lem}\label{unram2}
	Suppose $|a|\le 1$ and $\Re s>-1/2$, and set $b=a^{-1}$.
	When $|t|< |b|$,
	\begin{align*}
		B(t)=\begin{cases}2^{-1}|t|^{s+1}|b|^{-2s-1}(2^{2s+2}-1)(1-2^{-2s-1})^{-1}
			& (|b|=1), \\
			|t|^{s+1}|b|^{-2s-1}\{(2^{2s+1}-2^{-1})(1-2^{-2s-1})^{-1} +\delta(|t|=|2b|)(2^s-2^{2s+1})\}& (|b|>1).
		\end{cases}
	\end{align*}

	When $|b|\le |t|$, $B(t) = |t|^{-s}(1-2^{-2s-2})(1-2^{-2s-1})^{-1}.$
\end{lem}
\begin{proof}
	The same formula of $B_2(t)$ holds as in the proof of \cite[Lemma 9.11]{SugiyamaTsuzuki2018}.
	Let us consider $B_1(t)$.
	
	(i) Suppose $|t|<|b|$. Let $x \in D_1$.
	Then $f(x) = |t^{-1}b^2|^{-s-1}$ if $|x|<|b|$.
	
	If $|x|=|b|$ and $|t| \le |4b|$, then $|1\pm x| \le \max(1,|x|)=|b|$ and $|t^{-1}(b^2\tau -x^2)|=|4t^{-1}b^2|\ge |b|$.
	Thus $f(x) = |4t^{-1}b^2|^{-s-1}=2^{2s+2}|t|^{s+1}|b|^{-2s-2}$.
	
	If $|x|=|b|$ and $|t| =|2b|$, we observe $|t^{-1}(b^2\tau-x^2)| = |4t^{-1}b^2|=|t|$.
	When $|b|>1$, then $|1\pm x|=|x|=|b|>|t|$ holds, and whence
	$f(x) = |b|^{-s-1}$.
	When $|b|=1$, we have $1\pm x \in2\ZZ_2$ and $|t|=|2b|=2^{-1}$, whence  $f(x)=(2^{-1})^{-s-1}=|4t^{-1}b^2|^{-s-1}$.
	Hence, under $|b|=1$, we have
	$B_1(t)=2^{-1}|t|^{s+1}|b|^{-2s-1}(1+2^{2s+2}).$
	When $|b|>1$, we have 
	$B_1(t)=|b|^{-2s-1}|t|^{s+1}\{2^{-1}+2^{2s+1}+\delta(|t|=|2b|)(2^s-2^{2s+1})\}.$

	(ii) When $|t|\ge|b|$, then $|t|\ge1$ and $D_1=\{x \in \QQ_2 | |x|\le|t| \}$.
	Thus $|1\pm x|\le|t|$, $|t^{-1}(b^2\tau-x^2)|\le|t|$, and whence
	we have the same formula $B_1(t)=|t|^{-s}$ as in \cite[Lemma 9.11]{SugiyamaTsuzuki2018}.
	
	From the consideration above, we obtain the assertion.
\end{proof}

Combining \eqref{expression of integral} with Lemmas \ref{unram1} and \ref{unram2}, we have the following.
\begin{lem}Suppose $\Re s>(|\Re z|-1)/2$. If $|a|>1$,
	\begin{align*}
		\int_{\QQ_2^\times}B(t)|t|^{\frac{z-1}{2}}d^\times t =(1-2^{-s-1})\{\frac{\zeta_{2}(z)\zeta_{2}(s+\frac{-z+1}{2})}{L(\frac{z+1}{2},\varepsilon_{\tau})}
		+ \frac{\zeta_{2}(-z)\zeta_{2}(s+\frac{z+1}{2})}{L(\frac{-z+1}{2},\varepsilon_{\tau})}|b|^z2^{-z}\}.
	\end{align*}
	If $|a|\le 1$, then
	\begin{align*}\int_{t \in \QQ_2^\times} B(t)|t|^{\frac{z-1}{2}}d^\times t
		= |b|^{-s+\frac{z-1}{2}}(1-2^{-s-1})2^{s+\frac{-z+1}{2}}\{\frac{\zeta_{2}(s+\frac{z+1}{2})\zeta_{2}(s+\frac{-z+1}{2})}{L(s+1,\varepsilon_{\Delta})}-\delta(|b|>1)\}.
	\end{align*}
\end{lem}

From this, by noting $L(s,\varepsilon_{\tau})=(1+2^{-s})^{-1}$ and \eqref{phi(1)}, we have Theorem \ref{explicit dyadic} for $\tau=5$ when $a\neq 0$.
The case $a=0$ is treated in the same way as \cite[\S9.2.4]{SugiyamaTsuzuki2018}.
Consequently, we obtain the desired formula.

\medskip
\noindent
{\bf Remark} : By Theorem \ref{explicit dyadic},
we can generalize \cite{SugiyamaTsuzuki2018} to the case $\Sigma_{\rm dyadic}\cap S\neq \emptyset$:
Indeed, discussions in \S 7.3, \S 7.4 and Lemma 8.3 of \cite{SugiyamaTsuzuki2018} work even when
$\Sigma_{\rm dyadic}\cap S\neq \emptyset$.

\section*{Aknowledgements}

The authors would like to thank the anonymous referee for careful reading of the draft.
The first author was supported by
Grant-in-Aid for Research Activity Start-up 
18H05835.
The second author was supported by Grant-in-Aid for Scientific research (C) 15K04795.

%
%

\end{document}